\documentclass[11pt]{article}
\usepackage{amsmath, amsthm, amssymb}
\usepackage{enumerate}
\usepackage{mathrsfs}
\usepackage{color}
\usepackage[english]{babel}
\usepackage{graphicx}
\usepackage{float}
\usepackage{graphics}

\usepackage{multirow}
\usepackage{psfrag}
\usepackage[english]{babel}
\usepackage[ansinew]{inputenc}

\def\XXint#1#2#3{{\setbox0=\hbox{$#1{#2#3}{\int}$}
    \vcenter{\hbox{$#2#3$}}\kern-.5\wd0}}

\def\({\left(}
\def\){\right)}

\def\N{{\mathbb{N}}}
\def\RR{{\mathbb{R}}}
\def\CC{{\mathbb{C}}}

\newcommand{\be}{\begin{equation}}
\newcommand{\ee}{\end{equation}}
\newcommand{\bea}{\begin{eqnarray}}
\newcommand{\eea}{\end{eqnarray}}
\newcommand{\beann}{\begin{eqnarray*}}
\newcommand{\eeann}{\end{eqnarray*}}

\def\vp{\varphi}

\def\mH{\mathcal{H}}

\newcommand{\abs}[1]{\left\vert{#1}\right\vert}
\newcommand{\norm}[1]{\left\Vert{#1}\right\Vert}

 \newtheorem{Theorem}{Theorem}

\newtheorem{remark}{Remark}[section]

\newtheorem{Lemma}{Lemma}[section]

\renewcommand{\theTheorem}{\arabic{Theorem}}
\renewcommand{\theLemma}{\arabic{section}.\arabic{Lemma}}

\newcommand{\bn}{\begin{eqnarray*}}       
\newcommand{\en}{\end{eqnarray*}}
\newcommand{\beq}{\begin{equation}}
\newcommand{\eeq}{\end{equation}}

\newcommand{\Z}{\mathbb{Z}}
\newcommand{\R}{\mathbb{R}}

\numberwithin{equation}{section}

\begin{document}

\title{\bf Orbital Stability of Domain Walls \\ in Coupled Gross-Pitaevskii Systems}

\author{Andres Contreras$^1$, Dmitry E. Pelinovsky$^2$, and Michael Plum$^{3}$ \\
{\small $^{1}$ Department of Mathematical Sciences, New Mexico State University, Las Cruces,
New Mexico, USA }\\
{\small $^{2}$ Department of Mathematics and Statistics, McMaster
University, Hamilton, Ontario, Canada, L8S 4K1} \\
{\small $^{3}$ Institut f\"{u}r Analysis, Karlsruher Institut f\"{u}r Technologie, Karlsruhe, Germany, 76131 } }

\date{\today}
\maketitle

\begin{abstract}
Domain walls are minimizers of energy for coupled one-dimensional Gross--Pitaevskii systems
with nontrivial boundary conditions at infinity. It has been shown in \cite{ABCP} that
these solutions are orbitally stable in the space of complex $\dot{H}^1$
functions with the same limits at infinity. In the present work we adopt a new weighted $H^1$ space
to control perturbations of the domain walls and thus to obtain an improved orbital stability result.
A major difficulty arises from the degeneracy of linearized operators at the domain walls
and the lack of coercivity.
\end{abstract}

\section{Introduction}

{\em Domain walls} are heteroclinic connections for coupled two-component systems, for which the first component
connects zero and nonzero equilibria in the spatial domains, where the second one connects
the nonzero and zero equilibria respectively. Domain walls occur in many physical experiments, e.g.
in the convection patterns \cite{Malomed1,Malomed2},
nonlinear optics \cite{Sheppard1,Sheppard2}, two mixed Bose--Einstein condensates \cite{Barankov,Dror,Band},
and recently in immiscible binary Bose gases \cite{Fil1,Fil2}.
The existence and uniqueness of domain walls in the limits of strong and weak segregation was explored by means of rigorous asymptotic analysis
in \cite{AS,BLWZ,Merlet,Sourdis}. The existence, spectral and nonlinear orbital stability of domain walls was obtained from
a variational technique in \cite{ABCP} (see \cite{Berg} for earlier results).

In this work we are interested in obtaining a strengthened stability property of domain wall solutions.
To simplify our presentation, we consider the system of coupled cubic Gross-Pitaevskii (GP) equations
written in the form
\be\label{GP}
\left\{
\begin{matrix}
i\partial_t \psi_1=-\partial_x^2 \psi_1 + (\abs{\psi_1}^2+\gamma\abs{\psi_2}^2)\psi_1,\\
i\partial_t \psi_2=-\partial_x^2 \psi_2 + (\gamma\abs{\psi_1}^2+\abs{\psi_2}^2)\psi_2,
\end{matrix}
\right.
\ee
where $\gamma > 1$ is the coupling parameter. The system (\ref{GP}) is a particular case (but the most important one)
of the coupled GP systems, for which the results of \cite{ABCP} apply. Domain walls are special solutions
to the system (\ref{GP}) given by $\psi_{1,2}(t,x) := e^{-it} u_{1,2}(x)$, where the stationary profiles $u_{1,2}$
solve the following system of differential equations
\be\label{Evarphi}
\left\{
\begin{matrix}
-u_1''(x)+(|u_1|^2+\gamma |u_2|^2-1)u_1=0,\\
-u_2''(x)+(\gamma |u_1|^2+ |u_2|^2-1)u_2=0,
\end{matrix}
\right.
\ee
subject to the following boundary conditions at infinity
\begin{eqnarray}
\label{bc-ode}
\left\{ \begin{array}{l}
u_1(x)\to 0,\quad u_2(x)\to 1,\quad \mbox{ as }x\to -\infty,\\
u_1(x)\to 1,\quad u_2(x)\to 0,\quad \mbox{ as }x\to +\infty. \end{array} \right.
\end{eqnarray}

Existence of domain walls  for $\gamma>1$ has been shown in \cite{ABCP} by minimizing
the energy functional
\begin{equation}
\label{energy}
E(\Psi) := \int_{-\infty}^\infty \left(|\partial_x \psi_1|^2+|\partial_x \psi_2|^2 + \frac{1}{2}\left(|\psi_1|^2+|\psi_2|^2-1\right)^2 +(\gamma-1)|\psi_1|^2|\psi_2|^2\right) dx,
\end{equation}
in the class of functions in the energy space
\begin{equation}
\label{energy-space}
\mathcal{D} = \left\{ \Psi\in H_{loc}^1(\RR): \quad (|\psi_1(x)|,|\psi_2(x)|) \to e_{\pm} \quad {\rm as} \quad x \to \pm \infty \right\},
\end{equation}
where $e_+ = (1,0)$ and $e_- = (0,1)$. The energy space is equipped with the family of distances parameterized by $A > 0$:
\begin{equation}\label{distance}
\rho_A(\Psi,\Phi):= \sum_{j=1,2} \left[ \bigl\| \partial_x \psi_j - \partial_x \varphi_j \bigr\|_{L^2(\RR)}
  + \bigl\| |\psi_j|-|\varphi_j| \bigr\|_{L^2(\RR)} +   \bigl\| \psi_j -\varphi_j\bigr\|_{L^\infty(-A,A)}\right].
\end{equation}
By Theorems 2.1, 2.4, and 3.1 in \cite{ABCP}, the minimizers of energy (\ref{energy}) are
given by real solutions $U = (u_1, u_2) \in \mathcal{D}$ to the system \eqref{Evarphi}
up to the gauge translation. The profiles of the domain walls $u_1,u_2 \in \mathbb{R}$
satisfy the following properties:
\begin{itemize}
\item[(a)] $u_1 (x)=u_2 (-x)$ for all $x\in \mathbb{R}.$
\item[(b)] $u_1^2(x) + u_2^2(x) \leq 1$ for all $x\in \mathbb{R}.$
\item[(c)] $u_1'(x) > 0$ and $u_2'(x) < 0$ for all $x\in \mathbb{R}.$
\item[(d)] There are positive constants $C_-$ and $C_+$ such that
\begin{equation}
\label{exp-decay}
\left\{ \begin{array}{lr}
C_- e^{\sqrt{\gamma-1}x} \leq u_1 (x) \leq C_+ e^{\sqrt{\gamma-1}x}, \quad & x \leq 0, \\
C_- e^{-\sqrt{2} x} \leq 1 - u_1(x) \leq C_+ e^{-\sqrt{2} x}, \quad & x \geq 0. \end{array} \right.
\end{equation}
\end{itemize}
By Theorem 1.3 in \cite{AS}, the real minimizers of $E$ satisfying properties (a)--(c)
were shown to be the unique real solutions to the system (\ref{Evarphi}).

By the global well-posedness results in the energy space $\mathcal{D}$ in
\cite{Zhid}, for any $\Psi_0 \in \mathcal{D}\cap L^{\infty}(\RR)$,
there exists a unique global in time solution $\Psi(t) \in \mathcal{D}\cap L^{\infty}(\RR)$ to the coupled
GP system \eqref{GP} with initial data $\Psi(0)=\Psi_0$.
Moreover, the map $t \to \Psi(t)$ is continuous with respect to $\rho_A$ and the
energy of the coupled GP system \eqref{GP} is preserved along the flow, that is
$$
E(\Psi(t))=E(\Psi_0) \quad \mbox{\rm for all} \;\; t \in \mathbb{R}.
$$

Finally, by Theorems 1.4 and 1.5 in \cite{ABCP}, the following nonlinear orbital stability
theorem was established for the domain walls of the coupled GP system (\ref{GP}).

\begin{Theorem}[\cite{ABCP}]
Let $\Psi_0\in\mathcal{D}\cap L^\infty(\RR)$. There exists $A_0>0$ such that for any $A > A_0$
and for every $\varepsilon>0$, there exist a positive number $\delta>0$ and real functions
$\alpha(t), \theta_1(t), \theta_2(t)$ such that if
\[\rho_A(\Psi_0,U)\leq\delta,\]
then
\[\sup_{t\in\RR}\rho_A(\Psi(t),U_{\alpha(t),\theta_1(t),\theta_2(t)}) \leq \varepsilon,
\]
where $U_{\alpha(t),\theta_1(t),\theta_2(t)} = (e^{-i\theta_1(t)} u_1(\cdot - \alpha(t)), e^{-i\theta_2(t)} u_2(\cdot -\alpha(t)))$
is an orbit of domain walls. Moreover, there exists a positive constant $C$ such that for all $t\in\mathbb{R}$:
$$
\vert\alpha(t)\vert\leq C\varepsilon\max\{1,\vert t\vert\},
$$
provided $\varepsilon$ is sufficiently small. \label{Theorem-old}
\end{Theorem}

\begin{remark}
In Theorem \ref{Theorem-old}, modulation parameters $\theta_1$, $\theta_2$ for complex phases of $\psi_1$, $\psi_2$ are
not determined and are not controlled in the time evolution.
\end{remark}

The choice of the metric $\rho_A$ in (\ref{distance}) and the proof of Theorem \ref{Theorem-old}
were inspired by the similar results obtained for the nonlinear orbital stability
of black solitons in the cubic defocusing nonlinear Schr\"{o}dinger (NLS) equation in \cite{BGSS}.
On one hand, the domain walls are more complicated than black solitons because the gauge parameters
$\theta_1$ and $\theta_2$ have to be controlled separately from each other.
On the other hand, the domain walls are simpler than black solitons because the domain walls
are energy minimizers for the coupled GP system
whereas the black solitons are constrained energy minimizers for the NLS equation
under the constraint on the conserved renormalized momentum \cite{BGSS}.
Nevertheless, in both models, the principal difficulty in obtaining the nonlinear
orbital stability of black solitons or domain walls is the lack of coercivity of the energy
functional with respect to the imaginary parts of the perturbations.

Since the time of \cite{BGSS} and \cite{ABCP}, several important results have appeared
in the context of stability of the black solitons in the NLS equation. A new metric
has been introduced in \cite{OrbBlack} to obtain coercivity of the energy functional in the weighted $H^1$ space.
The new metric was introduced uniformly on the real line, so that the compact support
controlled by the parameter $A > 0$
in the family of distances $\rho_A$ in (\ref{distance}) becomes abundant. Once the coercivity of the energy
is obtained in the weighted $H^1$ space, nonlinear orbital stability and the asymptotic stability
of black solitons can be established by available analytical techniques in \cite{OrbBlack}.

The new variables introduced in \cite{OrbBlack} were further used in analysis of nonlinear orbital stability
of black solitons in the $H^2$ space by using a higher-order energy of the cubic NLS equation \cite{GPII}.
To tackle with the lack of coercivity, the family of distances given by (\ref{distance}) was still used and
analysis was developed separately inside and outside the compact support.
However, in the $H^2$ space, the black solitons become minimizers of the higher-order energy
and therefore, the constrained renormalized momentum is no longer needed to be used.

For completeness, we also mention other works on orbital and asymptotic stability of black solitons
in the cubic NLS equation, where more special studies are developed based on the inverse scattering transform method
\cite{Cuccagna,PZ}. However, this method is not applicable for the coupled GP system (\ref{GP}) unless
$\gamma = 1$, in which case no domain wall solutions exist.

The purpose of this work is to obtain improved nonlinear orbital stability results for the domain walls
of the coupled GP system compared to Theorem \ref{Theorem-old}. In this study, we incorporate the new weighted $H^1$ space
to control imaginary parts of the perturbations to the domain walls and to obtain the coercivity
of the energy functional uniformly on the real line. Due to nonlinear terms of the energy functional,
we are unable to control evolution of the real parts of the perturbations neither in the weighted $H^1$ space
nor in the standard $H^1$ space, in spite of the fact that the quadratic part of the energy functional is
coercive for the real parts in $H^1$. As a result, we have to introduce again the compact
support given by a parameter $R > 0$ and to control the real parts of the perturbations separately inside
and outside the compact support.

Following the approach of \cite{OrbBlack}, we introduce the new weighted $H^1$ space denoted by $\mH$,
according to the following inner product for $\Psi=(\psi_1,\psi_2)$ and $\Phi =(\vp_1,\vp_2)$:
\be
\label{weighted-H}
\langle \Psi, \Phi \rangle_{\mH}:= \sum_{j=1}^2 \int_\RR
\left[ \frac{d \psi_j}{dx} \frac{d \bar{\vp}_j}{dx} + (\gamma-1) (1- u_j^2) \psi_j \bar{\vp}_j \right] dx.
\ee
$\mathcal{H}$ is a Hilbert space and its squared induced norm is given by
\be
\label{norm-H}
\| \Psi \|_{\mathcal{H}}^2 := \langle \Psi, \Psi \rangle_\mH.
\ee
By property (b), the weight functions $1 - u_j^2$ are positive everywhere on the real line.
Also recall that $\gamma > 1$ so that the inner product does indeed yield a positive bilinear form.
Note that the Sobolev space $H^1(\RR)$ is continuously embedded into the weighted space $\mathcal{H}$ because there is a positive
constant $\mathcal{C}_H$ such that
\begin{equation}
\label{bound-H}
\| \Psi \|_{\mathcal{H}} \leq \mathcal{C}_H \| \Psi \|_{H^1}, \quad \mbox{\rm for every \;} \Psi \in H^1(\RR).
\end{equation}
Let us equip the space $\mathcal{H}$ with the family of distances parameterized by $R > 0$:
\begin{equation}\label{distance-R}
\rho_R(\Psi,\Phi):= \bigl\| \Psi - \Phi \bigr\|_{\mathcal{H}} + \sum_{j=1,2} \bigl\| |\psi_j|^2 - |\varphi_j|^2 \bigr\|_{L^2(|x| \geq R)}.
\end{equation}
The energy of perturbations to the domain walls turns out to be coercive in the metric $\rho_R$
for every $\gamma > 1$ and sufficiently large $R > 0$. The following theorem takes advantage of this coercivity
and gives an improved orbital stability result for the domain walls.

\begin{Theorem}
Let $\Psi_0 \in \mathcal{D}\cap L^\infty(\RR)$. There exists $R_0 > 0$
such that for any $R>R_0$ and for every $\varepsilon>0$, there exist a positive number $\delta>0$ and real functions
$\alpha(t), \theta_1(t), \theta_2(t)$ such that if
\begin{equation}
\label{bound-initial}
\rho_R(\Psi_0,U)\leq \delta,
\end{equation}
then
\begin{equation}
\label{bound-final}
\sup_{t\in\RR}\rho_R(\Psi(t),U_{\alpha(t),\theta_1(t),\theta_2(t)}) \leq \varepsilon,
\end{equation}
where $U_{\alpha(t),\theta_1(t),\theta_2(t)} = (e^{-i\theta_1(t)} u_1(\cdot - \alpha(t)), e^{-i\theta_2(t)} u_2(\cdot -\alpha(t)))$
is an orbit of domain walls. Moreover, there exists a positive constant $C$ such that for all $t\in\mathbb{R}$:
\begin{equation}
\label{bound-time-per}
\vert\alpha(t)\vert + \vert \theta_1(t)\vert + \vert \theta_2(t) \vert\leq C\varepsilon\max\{1,\vert t\vert\},
\end{equation}
provided $\varepsilon$ is sufficiently small. \label{Theorem-main}
\end{Theorem}

\begin{remark}
In Theorem \ref{Theorem-main}, modulation parameters $\alpha$, $\theta_1$, and $\theta_2$ are
uniquely determined by the projections in space $\mathcal{H}$ and are controlled in the time evolution
of the modulation equations.
\end{remark}

\begin{remark}
The proof of Theorem \ref{Theorem-main} is self-contained and it follows the ideas of the proof
of orbital stability of black solitons in \cite{GPII}, which are minimizers of the higher-order
energy of the nonlinear Schr\"{o}dinger equation in the $H^2$ space.
\end{remark}

\begin{remark}
As far as we can see, the distances $\rho_A$ and $\rho_R$ are not comparable: one can find examples of functions for which $\rho_A$ is finite while $\rho_R$ diverges and vice versa.
\end{remark}

The rest of this article is organized as follows. In
Section 2, we rewrite the energy functional given by (\ref{energy}) in terms
of perturbations to the domain walls. In Section 3, we prove coercivity of the energy functional
in the weighted space $\mathcal{H}$, provided $R > 0$ is sufficiently large.
Energy estimates are developed in Section 4. Modulation equations are analyzed in Section 5.
The proof of Theorem \ref{Theorem-main} is concluded in Section 6.
Appendix A describes an important technical result on continuation of
eigenvalues of the linearized operator with respect to the parameter $R > 0$ in the limit
$R \to \infty$.

\vspace{0.2cm}

{\bf Acknowledgements:} The work of A.C. was partially supported by a grant from the Simons Foundation \# 426318.

\section{Decomposition of the energy}

Let $U = (u_1, u_2) \in \RR^2$ be the domain wall solutions
to the ODE system (\ref{Evarphi}) subject to the boundary conditions (\ref{bc-ode}).
By adding a perturbation to $U$ and separating the real and imaginary parts
as $\Psi = U + V +i W$, we verify that the quadratic part of the energy
functional given by (\ref{energy}) can be block-diagonalized as follows:
\begin{equation}
\label{energy-quadratic-part}
E(U + V + i W) - E(U) = \left( L_+ V, V \right)_{L^2} + \left( L_- W, W \right)_{L^2}
+ \mathcal{O}(\| V + i W \|_{H^1}^3),
\end{equation}
where $L_{\pm} : H^2(\RR) \to L^2(\RR)$ are the linear self-adjoint operators given by
\be
L_{+} = \begin{bmatrix}
-\partial^2_x+3u_1^2+\gamma u_2^2-1 & 2\gamma u_1 u_2\\
2\gamma u_1 u_2 & -\partial^2_x +\gamma u_1^2+3u_2^2-1
\end{bmatrix} \label{L-plus}
\ee
and
\be
L_{-} = \begin{bmatrix}
-\partial^2_x +u_1^2+\gamma u_2^2-1 & 0\\
0 & -\partial^2_x+\gamma u_1^2+u_2^2-1
\end{bmatrix}. \label{L-minus}
\ee
By Theorem 3.1 in \cite{ABCP}, the linear operators $L_{\pm}$ satisfy the following properties:
\begin{enumerate}
\item[(i)] Each operator $L_+$ and $L_-$ is positive semi-definite on $H^1(\RR)$.
\item[(ii)]  Zero is a simple eigenvalue of $L_+$, with associated eigenfunction $\partial_x U$ and
there exists $\Sigma_0 > 0$ with $\sigma_{ess}(L_{+})=[\Sigma_0,\infty)$.
\item[(iii)]  $\sigma_{ess}(L_{-})=[0,\infty)$, and $L_- U_1 = L_- U_2 = 0$ with $U_1 = (u_1,0)$ and $U_2 = (0,u_2)$.
\end{enumerate}

By property (ii), the quadratic form for the operator $L_+$ is coercive in $H^1(\RR)$
under a single constraint which fixes the spatial translation
of the domain wall solutions. In other words, there exists a positive constant $\mathcal{C}_0$
such that
\begin{equation}
\label{coercivity-L-plus}
\left( L_{+} V, V \right)_{L^2} \geq \mathcal{C}_0 \norm{V}^2_{H^1}\quad \mbox{\rm for every \;\;} V \in H^1(\RR) : \quad
(V,\partial_x U)_{L^2}=0.
\end{equation}
On the other hand, since the essential spectrum of $L_-$ touches zero with two bounded wave functions
$U_1$ and $U_2$, which are not in $L^2(\RR)$,
the quadratic form for the operator $L_-$ is not coercive in $H^1(\RR)$. The same problem arises for black solitons of the cubic NLS equation
and it is dealt with the choice of quadratic variables which are not only bounded but also in $L^2(\RR)$ \cite{OrbBlack}.
Following this approach, we introduce the quadratic variables:
\begin{equation}
\label{variables-eta}
\eta_j := |u_j + v_j + i w_j|^2 - u_j^2 = 2 u_j v_j + v_j^2 + w_j^2
\end{equation}
By the explicit computations, we show that the energy functional given by (\ref{energy}) can be represented
in variables $V := (v_1,v_2)$, $W := (w_1,w_2)$, and $\Gamma := (\eta_1,\eta_2)$ as a sum of three quadratic forms.

\begin{Lemma}
\label{lem-decomposition}
Assume $V, W \in H^1(\RR)$. Then,
\begin{equation}
\label{energy-decomposition}
E(U + V + i W) - E(U) = \left( L_{-} V, V \right)_{L^2} + \left( L_{-} W, W \right)_{L^2} + \frac{1}{2} \left( M \Gamma,\Gamma \right)_{L^2},
\end{equation}
where $L_-$ is given by (\ref{L-minus}) and
\[
M = \begin{bmatrix}
1 & \gamma\\
\gamma & 1
\end{bmatrix}.\]
\end{Lemma}

\begin{proof}
By substituting the perturbation $V+iW$ to the domain wall solution $U$ in the component form, we obtain
\begin{eqnarray*}
E(U + V + i W) & = & \int_{\RR} \left[ \abs{u_1' + v_1' + i w_1'}^2+\abs{u_2' + v_2' + i w_2'}^2  \right. \\
& \phantom{t} & \left. + \frac{1}{2} \left(1-\abs{u_1 + v_1 + i w_1}^2-\abs{u_2 + v_2 + i w_2}^2\right)^2 \right. \\
& \phantom{t} & \left. + (\gamma - 1) \abs{u_1 + v_1 + i w_1}^2  \abs{u_2 + v_2 + i w_2}^2 \right] dx.
\end{eqnarray*}
Since the domain wall $U$ is a critical point of $U$, the linear terms in $V$ and $W$ are canceled after integration by parts.
By subtracting $E(U)$ from $E(U+V+iW)$, we rewrite the result in the explicit form:
\begin{eqnarray*}
& \phantom{t} & \int_{\RR}
\left[ (v_1')^2 + (w_1')^2 + (v_2')^2 + (w_2')^2 - (1-u_1^2-u_2^2)(v_1^2+w_1^2+v_2^2+w_2^2) \right. \\
& \phantom{t} & \left. + \frac{1}{2} (\eta_1+\eta_2)^2
+ (\gamma-1) u_1^2 (v_2^2+w_2^2) + (\gamma-1) u_2^2 (v_1^2+w_1^2) + (\gamma-1) \eta_1 \eta_2  \right] dx.
\end{eqnarray*}
Rewriting this expression in the matrix-vector form and canceling similar terms yield (\ref{energy-decomposition}).
\end{proof}

\begin{remark}
Note that
\begin{equation}
\label{equivalent}
\left( L_{-} V,V \right)_{L^2} + 2 \left( M UV, UV \right)_{L^2} \equiv \left( L_{+} V,V \right)_{L^2},
\end{equation}
where $UV = (u_1 v_1, u_2 v_2)$ is understood in the component form. The quadratic part of (\ref{energy-decomposition})
with the substitution (\ref{variables-eta}) coincides with the quadratic part of (\ref{energy-quadratic-part}).
\label{rem-equivalent}
\end{remark}

\begin{remark}
\label{rem-coercivity}
By property (i) and (iii), the first two terms in (\ref{energy-decomposition}) are positive semi-definite
in $H^1(\RR)$. However, the third term is sign-indefinite, since $\gamma > 1$. If the equivalence (\ref{equivalent})
is used, the quadratic forms involving $V$ and $W$ are again positive semi-definite in $H^1(\RR)$
by property (i) but the energy decomposition (\ref{energy-decomposition}) includes also cubic and quartic terms
in $V$ and $W$. Due to the lack of coercivity for the operator $L_-$ in $H^1(\RR)$, the cubic and quartic terms
in $W$ cannot be controlled in $H^1(\RR)$.
\end{remark}

\section{Coercivity in a weighted $H^1$ space}

In order to deal with the poor coercivity of $L_-$ mentioned in Remark \ref{rem-coercivity},
we introduce the weighted space $\mathcal{H}$ given by the inner product (\ref{weighted-H})
and the squared norm (\ref{norm-H}).

By explicit computation, we have
\begin{equation}
\label{operator-L-T}
\left( L_{-} \Psi, \Psi \right)_{L^2} = \| \Psi \|_{\mathcal{H}}^2 - \gamma \langle T \Psi, \Psi \rangle_\mH,
\end{equation}
where $T : \mathcal{H} \to \mathcal{H}$ is the operator defined by the bilinear form
\begin{equation}
\label{operator-T}
\langle T \Psi, \Phi \rangle_\mH := \int_\RR \left(1-u_1^2-u_2^2 \right) \left( \psi_1 \bar{\vp}_1 +
\psi_2 \bar{\vp}_2 \right) dx.
\end{equation}
By properties (b) and (d), the weight function $1-u_1^2-u_2^2$ is positive and
decays to zero at infinity exponentially fast.
By using the representation (\ref{operator-L-T}) and (\ref{operator-T}),
we prove the following result on coercivity of $L_-$ in metric space $\mathcal{H}$
subject to the appropriate orthogonality conditions.

\begin{Lemma}
\label{lem-coercivity-L-minus}
There exists $\Lambda_- > 0$ such that
\begin{equation}
\label{coercivity-L-minus}
\left( L_{-} \Psi, \Psi \right)_{L^2} \geq \Lambda_- \norm{\Psi}_{\mathcal{H}}^2 \quad \mbox{ for  every } \Psi \in \mathcal{H} : \quad
\langle T \Psi, U_1 \rangle_\mH = \langle T \Psi, U_2 \rangle_\mH = 0,
\end{equation}
where $U_1 = (u_1,0)$ and $U_2 = (0,u_2)$.
\end{Lemma}

\begin{proof}
Thanks to the fast (exponential) decay of $1-u_1^2-u_2^2$ to zero at infinity,
the same arguments as in \cite{OrbBlack} imply that the operator $T$ is compact in $\mathcal{H}$, so that
its spectrum in $\mathcal{H}$ is purely discrete. Therefore, the spectrum of the operator
$\mathcal{L}_- := I - \gamma T$ in $\mathcal{H}$ consists of isolated eigenvalues $\lambda$ accumulating to
the point $\lambda_0 = 1$. Moreover, thanks to the positivity of $1 - u_1^2 - u_2^2$, the eigenvalues accumulate
to the point $\lambda_0 = 1$ from below.

By property (iii), the zero eigenvalue of $\mathcal{L}_-$ is at least double with
eigenvectors $U_1, U_2 \in \mathcal{H}$. To show that $U_{1,2}$ belong to $\mathcal{H}$, we note that
$\partial_x u_{1,2}$ decays exponentially at both infinities, whereas
$u_{1,2}$ decays exponentially to zero at the same infinity where $1-u_{1,2}^2$ is nonzero and vice versa.
Therefore, $U_{1,2} \in \mathcal{H}$.

Eigenvalues $\lambda$ of $\mathcal{L}_-$ in $\mathcal{H}$ are determined by the spectral problem
\begin{eqnarray}
\label{spectralproblem}
L_- \Psi = \lambda K \Psi, \quad \Psi \in \mathcal{H},
\end{eqnarray}
where $L_-$ is given by (\ref{L-minus}) and $K$ is given by
\begin{equation}
\label{K}
K = \begin{bmatrix}
-\partial^2_x + (\gamma-1)(1- u_1^2) & 0\\
0 & -\partial^2_x + (\gamma-1)(1- u_2^2)
\end{bmatrix}.
\end{equation}
Note that $(K \Psi, \Psi)_{L^2} = \| \Psi \|_{\mathcal{H}}^2$.

We note that $L_-$ and $K$ are diagonal operators consisting of two Schr\"{o}dinger operators.
As a result, the spectral problem (\ref{spectralproblem}) can be written separately
for each Schr\"{o}dinger operator as follows:
\begin{eqnarray}
\label{Schr1}
\left[ -\partial_x^2 + u_1^2 + \gamma u_2^2 - 1 \right] \psi_1 = \lambda \left[-\partial^2_x + (\gamma-1)(1- u_1^2)\right] \psi_1
\end{eqnarray}
and
\begin{eqnarray}
\label{Schr2}
\left[ -\partial_x^2 + \gamma u_1^2 +  u_2^2 - 1 \right] \psi_2 = \lambda \left[-\partial^2_x + (\gamma-1)(1- u_2^2)\right] \psi_2.
\end{eqnarray}
Each Schr\"{o}dinger equation (\ref{Schr1}) and (\ref{Schr2}) with $\lambda = 0$ has one bounded and
one unbounded linearly independent solutions. The unbounded solution grows exponentially at the same infinity
where the bounded solution decays exponentially because the Wronskian of the two linearly independent solutions
is constant and nonzero. Since the unbounded solution does not belong to the space $\mathcal{H}$,
the kernel of $\mathcal{L}_-$ in $\mathcal{H}$ is two-dimensional, spanned by $U_1$ and $U_2$.

Next, we show that the spectral problem (\ref{spectralproblem}) has no negative eigenvalues $\lambda$.
Indeed, the Schr\"{o}dinger equation (\ref{Schr1}) for $\lambda < 1$ can be rewritten in the form
\begin{equation}
\label{Schr3}
\left[ - \partial_x^2 + P(x;\lambda) \right] \psi_1 = 0, \quad
P(\cdot;\lambda) := \frac{u_1^2+\gamma u_2^2 - 1 + \lambda (\gamma-1) (u_1^2 - 1)}{1-\lambda},
\end{equation}
where $P(x;\lambda)$ satisfies
$$
\partial_{\lambda} P(x;\lambda) = - \frac{\gamma (1 - u_1(x)^2 - u_2(x)^2)}{(1-\lambda)^2} \leq 0, \quad x \in \mathbb{R},
$$
where the last inequality follows by property (b). By Sturm's Comparison Theorem,
the Schr\"{o}dinger operator $-\partial_x^2 + P(x;\lambda_1): H^2(\mathbb{R}) \to L^2(\mathbb{R})$
may have zero resonance (a bounded solution in $L^{\infty}(\mathbb{R})$ for zero eigenvalue)
only if the Schr\"{o}dinger operator $-\partial_x^2 + P(x;\lambda_2): H^2(\mathbb{R}) \to L^2(\mathbb{R})$
with $\lambda_2 > \lambda_1$ has a negative eigenvalue (a localized solution in $L^2(\mathbb{R})$ for a negative eigenvalue).
Since the Schr\"{o}dinger operator $-\partial_x^2 + P(x;0): H^2(\mathbb{R}) \to L^2(\mathbb{R})$.
has no negative eigenvalues, the Schr\"{o}dinger equation
(\ref{Schr3}), or equivalently, the Schr\"{o}dinger equation (\ref{Schr1}), admits no bounded solutions in $L^{\infty}(\mathbb{R})$
for $\lambda < 0$. A similar argument applies to the Schr\"{o}dinger equation
(\ref{Schr2}) for $\lambda < 0$.

Thus, the spectral problem (\ref{spectralproblem}) has no negative eigenvalues $\lambda$,
whereas the zero eigenvalue is double and isolated from the rest of the spectrum of $\mathcal{L}_-$ in $\mathcal{H}$.
The next (nonzero) eigenvalue of $\mathcal{L}_-$ in $\mathcal{H}$ is positive.
Let us denote the positive eigenvalue by $\Lambda_-$.
Since the zero eigenvalue is exactly double, the orthogonality conditions
\[
0 = \langle \Psi, U_{1,2} \rangle_\mH = \gamma \int_{-\infty}^\infty(1-u_1^2-u_2^2) u_{1,2} \psi_{1,2} dx
= \gamma \langle T \Psi, U_{1,2} \rangle_\mH
\]
remove projections to the eigenvectors $U_1$ and $U_2$. The coercivity
bound (\ref{coercivity-L-minus}) holds by the standard spectral theorem in $\mathcal{H}$.
\end{proof}

\begin{remark}
By Lemma \ref{lem-coercivity-L-minus}, the imaginary part of perturbations $W$ to the domain walls $U$
is well controlled in the metric space $\mathcal{H}$ subject to the two orthogonality conditions (\ref{coercivity-L-minus})
that specify complex phases of the two components of $\Psi = (\psi_1,\psi_2)$ due to gauge rotations.
Compared to \cite{OrbBlack},  no additional orthogonality conditions are needed because the
domain wall solutions are true minima of the energy functional $E$.\label{remark-imaginary}
\end{remark}

\begin{remark}
For the real part of perturbations $V$ to the domain walls $U$,
we can only add one orthogonality condition that specifies a spatial translation of the solution $\Psi$.
If we add the condition $(V, \partial_x U)_{L^2} = 0$, the coercivity
bound (\ref{coercivity-L-plus}) in $H^1$ is not useful because if
$\Gamma \in L^2(\mathbb{R})$, $W \in L^{\infty}(\RR)$ but $W \notin L^2(\RR)$,
then there is no way that $V \in L^2(\RR)$. On the other hand, even though
$( L_{-} V, V)_{L^2} \geq 0$, the last term in the decomposition (\ref{energy-decomposition})
is sign-indefinite if $\gamma>1$ and hence the coercivity to control the real part of
perturbations $V$ is lost. \label{remark-real}
\end{remark}

To handle the problem described in Remark \ref{remark-real}, we divide the real line into three regions
$(-\infty,-R)$, $[-R,R]$, and $(R,\infty)$ for a fixed $R > 0$. We further introduce
a family of linear operators that interpolate between $L_-$ as $R \to 0$
and $L_+$ as $R \to \infty$. The family is given explicitly by
\begin{eqnarray}
\nonumber
L_R&=&L_{-}+2\begin{bmatrix} u_1^2 & \gamma u_1 u_2\\  \gamma u_1 u_2 & u_2^2 \end{bmatrix}\chi_{[-R,R]}\\
&=&L_{+}-2\begin{bmatrix} u_1^2 & \gamma u_1 u_2 \\  \gamma u_1 u_2 & u_2^2 \end{bmatrix}\chi_{(-\infty, -R)\cup(R, \infty)},
\label{R}
\end{eqnarray}
where $\chi$ is the characteristic function.
Using the same metric space $\mathcal{H}$ as is given by the inner product (\ref{weighted-H})
and the squared norm (\ref{norm-H}), we obtain
\begin{equation}
\label{operator-L-R-T}
\left( L_R \Psi, \Psi \right)_{L^2} = \| \Psi \|_{\mathcal{H}}^2 - \langle T_R \Psi, \Psi \rangle_\mH,
\end{equation}
where $T_R : \mathcal{H} \to \mathcal{H}$ is the operator defined by the bilinear form
\begin{eqnarray}
\nonumber
\langle T_R \Psi, \Phi \rangle_\mH & := & \gamma
\int_\RR \left(1-u_1^2-u_2^2 \right) \left( \psi_1 \bar{\varphi}_1 +
\psi_2 \bar{\varphi}_2 \right) dx \\
& \phantom{t} & -
2 \int_{-R}^R \left( u_1^2 \psi_1 \bar{\varphi}_1 + \gamma u_1 u_2 (\psi_1 \bar{\varphi}_2 + \psi_2 \bar{\varphi}_1)
+ u_2^2 \psi_2 \bar{\varphi}_2 \right) dx.
\label{operator-R-T}
\end{eqnarray}

In Appendix A, we prove Theorem \ref{Theo1} which states continuity with respect to $R$ of eigenvalues of the operator
$\mathcal{L}_R := I - T_R$ in $\mathcal{H}$ below the level $\lambda_0 = 1$. As $R \to \infty$,
eigenvalues of $\mathcal{L}_R$ converge to the eigenvalues of the operator
$\mathcal{L}_+ := I - T_{\infty}$ in $H^1(\mathbb{R})$ below the level $\lambda_0 = 1$.
By using this continuation, we prove the following result on coercivity of the operator $L_R$ in metric space
$\mathcal{H}$ subject to a single orthogonality condition.

\begin{Lemma}
\label{lem-coercivity-L-R}
There exists $R_0 > 0$ and $\Lambda_+ > 0$ such that for any $R > R_0$,
\begin{equation}
\label{coercivity-L-R}
\left( L_R \Psi, \Psi \right)_{L^2} \geq \Lambda_+ \norm{\Psi}_{\mathcal{H}}^2 \quad
{\rm for \;\; every \;} \Psi \in \mathcal{H} : \quad
\langle \Psi, \partial_x U \rangle_\mH = 0.
\end{equation}
\end{Lemma}

\begin{proof}
Thanks to the fast (exponential) decay of $1-u_1^2-u_2^2$ to zero at infinity
and the compact support of the second integral in (\ref{operator-R-T}),
the operator $T_R$ for any fixed $R > 0$ is compact in $\mathcal{H}$.
Therefore, the spectrum of the operator $\mathcal{L}_R := I - T_R$ in $\mathcal{H}$
consists of isolated eigenvalues accumulating to the point $\lambda_0 = 1$. Eigenvalues
$\lambda$ of $\mathcal{L}_R$ in $\mathcal{H}$ are given by the spectral problem
\begin{equation}
\label{spectral-problem-R}
L_R \Psi = \lambda K \Psi, \quad \Psi \in \mathcal{H},
\end{equation}
where $L_R$ is given by (\ref{R}) and $K$ is given by (\ref{K}).

In comparison, the operator $T_{\infty}$ for $R = \infty$ is not compact in $\mathcal{H}$, so that
the spectrum of $\mathcal{L}_+ := I - T_{\infty}$ is only defined in $H^1(\mathbb{R})$ and
includes also the essential spectrum bounded from below by $\lambda_0 = 1$.
The spectrum of $\mathcal{L}_+$ is defined by the spectral problem
\begin{equation}
\label{spectral-problem-L-plus}
L_+ \Psi = \lambda K \Psi, \quad \Psi \in H^1(\mathbb{R}),
\end{equation}
where $L_+$ is given by (\ref{L-plus}) and $K$ is given by (\ref{K}).
From the asymptotic values of the potentials of $L_+$ and $K$ at infinity,
we can see that the essential spectrum of the spectral problem (\ref{spectral-problem-L-plus})
is located for $\lambda \in [1,\infty)$. By property (i), isolated eigenvalues
of the spectral problem (\ref{spectral-problem-L-plus}) are located for $\lambda \in [0,1)$.

The first (zero) eigenvalue of $\mathcal{L}_+$ is simple and
corresponds to the eigenvector $\Psi = \partial_x U$ by property (ii).
The second (nonzero) eigenvalue of the spectral problem (\ref{spectral-problem-L-plus}) is strictly positive.
The coercivity bound (\ref{coercivity-L-R}) for a fixed $R > 0$ sufficiently large is obtained by
continuity of isolated eigenvalues of the operator $\mathcal{L}_R$ in $\mathcal{H}$
below the point $\lambda_0 = 1$ with respect to the parameter $R$.
The continuity of eigenvalues below $\lambda_0 = 1$ as $R \to \infty$ is proved in
Theorem \ref{Theo1} of Appendix A.
\end{proof}

\begin{remark}
As $R \to 0$, $\mathcal{L}_R$ converges to $\mathcal{L}_-$ in the norm of $\mathcal{H}$, hence
$\mathcal{L}_R$ has two small eigenvalues for small $R > 0$. The coercivity bound (\ref{coercivity-L-R})
with a single orthogonality condition only holds for large $R > 0$ and clearly fails as $R \to 0$.
\end{remark}

\section{Energy estimates}

As we described in Remark \ref{remark-real}, the decomposition (\ref{energy-decomposition}) in Lemma \ref{lem-decomposition}
for the difference between energy levels is not really useful. On the other hand, the equivalent
representation (\ref{equivalent}) in Remark \ref{rem-equivalent} can not be used uniformly on
the real line. Due to these reasons, we write
\begin{eqnarray}
E(U + V + iW) - E(U) = \Delta E + \left( L_{-} W, W \right)_{L^2}
\label{energy-levels}
\end{eqnarray}
where $\Delta E$ can be represented as follows:
\begin{eqnarray*}
\Delta E & = & \left( L_{-} V, V \right)_{L^2} + \frac{1}{2} \left( M \Gamma,\Gamma \right)_{L^2} \\
& = & \int_{-R}^R B_{+}(V) dx + \int_{-R}^R \left[ N_3(V,W) + N_4(V,W) \right] dx \\
& \phantom{t} & +
\left( \int_{-\infty}^{-R} + \int_R^\infty \right) B_{-}(V) dx +
\frac{1}{2} \left( \int_{-\infty}^{-R} + \int_R^\infty \right) (\eta_1^2+\eta_2^2)dx \\
& \phantom{t} & + \gamma \int_{-\infty}^{-R} \eta_2( 2 u_1 v_1 + v_1^2+w_1^2) dx
+ \gamma \int^{\infty}_{R} \eta_1(2 u_2 v_2 + v_2^2+w_2^2) dx.
\end{eqnarray*}
Here $B_+(V)$ and $B_-(V)$ are densities for the quadratic forms
$\left( L_+ V, V \right)_{L^2}$ and $\left( L_{-} V, V \right)_{L^2}$,
whereas $N_3$ and $N_4$ are cubic and quartic terms given by
\begin{eqnarray*}
N_3(V,W) = 2 (v_1^2+w_1^2)(u_1 v_1 + \gamma u_2 v_2) + 2 (v_2^2+w_2^2) (\gamma u_1 v_1 + u_2 v_2)
\end{eqnarray*}
and
\begin{eqnarray*}
N_4(V,W) = \frac{1}{2} \left[ (v_1^2+w_1^2)^2 +
2 \gamma (v_1^2+w_1^2) (v_2^2+w_2^2) + (v_2^2+w_2^2)^2 \right].
\end{eqnarray*}
The quadratic part given by $B_+(V)$ and $B_-(V)$ is represented by the
quadratic form associated with the operator $L_R$ defined by (\ref{R}). Thus,
the representation for $\Delta E$ is different on the intervals $(-\infty,-R)$, $[-R,R]$, and $(R,\infty)$.

Let us consider estimates on the semi-infinite interval $[R,\infty)$.
Since  $u_2(x)$ is exponentially small as $x \to +\infty$, according to the
sharp decay estimates (\ref{exp-decay}),
it follows from the definitions (\ref{weighted-H}) and (\ref{norm-H}) that
\be
\label{embeddings}
\norm{v_2}_{H^1(R,\infty)}  \leq C_{\gamma} \| V \|_{\mathcal{H}}, \quad \norm{w_2}_{H^1(R,\infty)}\leq C_{\gamma} \| W \|_{\mathcal{H}},
\ee
for some positive constant $C_{\gamma}$ that depends on $\gamma > 1$.
By Sobolev's embedding, we have $v_2(x) + i w_2(x) \to 0$ as $x \to +\infty$ and
\be
\label{embed-L-inf}
\norm{v_2 + i w_2}_{L^{\infty}(R,\infty)} \leq
C_{\rm emb} \norm{v_2 + i w_2}_{H^1(R,\infty)} \leq C_{\rm emb} C_{\gamma} \| V + i W \|_{\mathcal{H}},
\ee
where $C_{\rm emb}$ is the Sobolev embedding constant. In what follows, we will
omit writing the dependence
of the positive constants from the fixed parameter $\gamma > 1$.

The estimate (\ref{embed-L-inf}) allows us to control the last term in $\Delta E$.
Since $u_2(x)$ is exponentially small as $x \to +\infty$
in accordance with (\ref{exp-decay}),
there are positive constants $C$ and $\kappa$ such that
\be
\label{estimates-1}
\left| \int^{\infty}_{R} \eta_1(2 u_2 v_2 + v_2^2+w_2^2)  dx \right| \leq
C \left( e^{-\kappa R} \| V + iW \|_{\mathcal{H}} + \| V + i W \|_{\mathcal{H}}^2 \right) \| \eta_1 \|_{L^2(|x| \geq R)}.
\ee
Similar estimates are available for the term
\be
\label{estimates-1-other}
\left| \int_{-\infty}^{-R} \eta_2( 2 u_1 v_1 + v_1^2+w_1^2) dx \right| \leq
C \left( e^{-\kappa R} \| V + iW \|_{\mathcal{H}} + \| V + i W \|_{\mathcal{H}}^2 \right) \| \eta_2 \|_{L^2(|x| \geq R)}.
\ee
since $u_1(x)$ is exponentially small as $x \to -\infty$.

It remains to control the nonlinear terms
$$
\int_{-R}^R \left[ N_3(V,W) + N_4(V,W) \right] dx.
$$
The quartic term $N_4$ is positive, therefore,
it is controlled from below by zero. The cubic term $N_3$ is bounded by
\be
\label{estimates-2}
\left| \int_{-R}^R N_3(V,W) dx \right| \leq C \| V + i W\|_{H^1(-R,R)}^3
\ee
for some positive constant $C$. However, since $1 - u_2^2(x)$ is exponentially small as $x \to -\infty$
and $1 - u_1^2(x)$ is exponentially small as $x \to +\infty$ in accordance with (\ref{exp-decay}), it follows
from the definitions (\ref{weighted-H}) and (\ref{norm-H}) that
\be
\label{embeddings-R}
\| V + i W \|_{H^1(-R,R)}  \leq C e^{\kappa R} \| V + i W \|_{\mathcal{H}},
\ee
for some positive constants $C$ and $\kappa$.

By combining (\ref{coercivity-L-minus}), (\ref{coercivity-L-R}), (\ref{energy-levels}), (\ref{estimates-1}),
(\ref{estimates-1-other}),
(\ref{estimates-2}), and (\ref{embeddings-R}), we obtain the estimate
\begin{eqnarray}
\nonumber
E(U + V + iW) - E(U) & \geq & \Lambda_+ \| V \|_{\mathcal{H}}^2 +
\Lambda_- \| W \|_{\mathcal{H}}^2 + \frac{1}{2} \norm{\eta_1}_{L^2(\abs{x}\geq R)}^2 +
\frac{1}{2} \norm{\eta_2}_{L^2(\abs{x}\geq R)}^2 \\
\nonumber
& \phantom{t} & - \gamma C e^{-\kappa R} \| V + i W \|_{\mathcal{H}} \left( \norm{\eta_1}_{L^2(\abs{x}\geq R)} +
\norm{\eta_2}_{L^2(\abs{x}\geq R)} \right) \\
\nonumber
& \phantom{t} &  -
\gamma C \| V + i W \|_{\mathcal{H}}^2 \left( \norm{\eta_1}_{L^2(\abs{x}\geq R)} +
\norm{\eta_2}_{L^2(\abs{x}\geq R)} \right) \\
& \phantom{t} &  -  C e^{3 \kappa R} \| V + i W \|^3_{\mathcal{H}},
\label{estimate}
\end{eqnarray}
provided $W$ satisfies the two orthogonality conditions in (\ref{coercivity-L-minus}) and
$V$ satisfies the only orthogonality condition in (\ref{coercivity-L-R}).
The latter constraints are satisfied by adding modulation parameters to the solution $\Psi$,
see Section \ref{sec-modulation}.

Let $\nu > 0$ be a small number that defines radius of a ball in $\mathcal{H}$ for the perturbation terms such that
\be
\label{ball}
\| V + i W \|_{\mathcal{H}} + \norm{\eta_1}_{L^2(\abs{x}\geq R)} + \norm{\eta_2}_{L^2(\abs{x}\geq R)} \leq \nu e^{-3 \kappa R}.
\ee
Note that the ball is exponentially small in terms of large parameter $R$.
Also note that the definition (\ref{ball}) agrees with the distance $\rho_R$ defined by (\ref{distance-R}).

For $\nu > 0$ sufficiently small and $R > 0$ sufficiently large, the estimate (\ref{estimate})
allows us to control the perturbation term
in terms of the conserved energy  by
\begin{eqnarray}
\| V + i W \|_{\mathcal{H}}^2 + \norm{\eta_1}_{L^2(\abs{x}\geq R)}^2 + \norm{\eta_2}_{L^2(\abs{x}\geq R)}^2
\leq C \left[ E(U + V + i W) - E(U)  \right].
\label{control}
\end{eqnarray}
The right-hand side of (\ref{control}) is conserved in time, so its value is defined by the initial data for the perturbation terms
$V + i W$. The estimates (\ref{ball}) and (\ref{control}) are compatible if
\begin{equation}
E(U + V + i W) - E(W) \leq C \nu^2 e^{-6 \kappa R}.
\label{inital-energy}
\end{equation}
The bound (\ref{inital-energy}) can be satisfied by the choice of $\delta$
in the bound (\ref{bound-initial}) for the initial data. Then, the bounds (\ref{control}) and (\ref{inital-energy})
are used to control
the solution over all times and to define $\varepsilon$ in the bound (\ref{bound-final}).

\section{Modulation equations}
\label{sec-modulation}

It remains to define a suitable solution $\Psi$ to the coupled GP system (\ref{GP}),
which can be decomposed as $U + V + i W$, where $W$ satisfies the two orthogonality conditions in (\ref{coercivity-L-minus}) and
$V$ satisfies the only orthogonality condition in (\ref{coercivity-L-R}).
This is done by introducing the modulation parameters $\alpha$, $\theta_1$, and $\theta_2$,
using the translation and gauge invariance in the coupled GP system (\ref{GP}),
and setting the modulation equations. The algorithm is fairly standard, see, e.g. the recent work in
\cite{GPII}, hence we overview only basic details of the algorithm. We note however that
the orthogonality conditions are formulated in the weighted space $\mathcal{H}$, which is adjusted to
the definition of the domain walls $U$. Therefore, one needs to be careful with the
definition of the modulation parameter $\alpha$.

We start by writing the solution to the coupled GP system (\ref{GP}) in the form
\begin{eqnarray}
\nonumber
\Psi_{\alpha(t),\theta_1(t),\theta_2(t)}(t,x) & := & (e^{it + i\theta_1(t)} \psi_1(t,x + \alpha(t)),
e^{i t + i\theta_2(t)} \psi_2(t,x +\alpha(t))) \\
& = & U(x) + V(t,x) + i W(t,x),
  \qquad (t,x) \in \R \times \R, \label{decomposition2}
\end{eqnarray}
where the perturbations $V$ and $W$ are real-valued and satisfy the
orthogonality conditions
\begin{equation}
\label{projections2}
  \langle V(t,\cdot), \partial_x U \rangle_{\mathcal{H}} \,=\, 0, \qquad
  \langle W(t,\cdot), U_1 \rangle_{\mathcal{H}} = \langle W(t,\cdot), U_2 \rangle_{\mathcal{H}} = 0, \qquad
  t \in \R.
\end{equation}
The constraints (\ref{projections2}) allow us to determine uniquely the
modulation parameters, namely the translation $\alpha(t)$ and the
complex phases $\theta_1(t)$ and $\theta_2(t)$, at least for solutions $\Psi(t,\cdot)$ in a small
neighborhood of the domain walls $U$. This is done according to the following lemma.

\begin{Lemma}
\label{Lemma-xith}
There exists $\varepsilon_0 > 0$ such that, for any $\Psi \in \mathcal{D}\cap L^\infty(\RR)$ satisfying
\begin{equation}
\label{inf2}
  \inf_{\alpha, \theta_1, \theta_2 \in \R} \| \Psi_{\alpha,\theta_1,\theta_2} - U \|_{\mathcal{H}} \,\le\, \varepsilon_0,
\end{equation}
there exist $\alpha \in \R$, $\theta_1 \in \R/(2\pi\Z)$, and $\theta_2 \in \R/(2\pi\Z)$ such that
\begin{equation}
\label{decomp2}
  \Psi_{\alpha,\theta_1,\theta_2} \,=\, U + V + i W,
\end{equation}
where the real-valued functions $V$ and $W$ satisfy the orthogonality conditions
\begin{equation}
\label{projections2stat}
  \langle V, \partial_x U \rangle_{\mathcal{H}} \,=\, 0, \qquad
  \langle W, U_1 \rangle_{\mathcal{H}} = \langle W, U_2 \rangle_{\mathcal{H}} = 0.
\end{equation}
Moreover, the modulation parameters $\alpha \in \R$, $\theta_1 \in \R/(2\pi\Z)$, and $\theta_1 \in \R/(2\pi\Z)$ depend continuously on
$\Psi$ in $\mathcal{H}$.
\end{Lemma}

\begin{proof}
It is sufficient to prove \eqref{decomp2} for all $\Psi \in \mathcal{D}\cap L^\infty(\RR)$
such that $\varepsilon := \| \Psi - U \|_{\mathcal{H}}$ is sufficiently small. Given
such a $\Psi \in \mathcal{D}\cap L^\infty(\RR)$, we consider the smooth vector field ${\bf f} : \R^3
\to \R^3$ defined by
\[
  {\bf f}(\alpha,\theta_1,\theta_2) \,=\, \begin{pmatrix} \langle {\rm Re} \Psi_{\alpha,\theta_1,\theta_2}, \partial_x U \rangle_{\mathcal{H}} \\
  \langle {\rm Im} \Psi_{\alpha,\theta_1,\theta_2}, U_1 \rangle_{\mathcal{H}} \\
  \langle {\rm Im} \Psi_{\alpha,\theta_1,\theta_2}, U_2 \rangle_{\mathcal{H}}
  \end{pmatrix}, \qquad (\alpha,\theta_1,\theta_2) \in \R^3.
\]
We check that $\langle U, \partial_x U \rangle_{\mathcal{H}} = 0$ by direct substitution in (\ref{weighted-H}) and integration.
Therefore, by construction, we have ${\bf f}(\alpha,\theta_1,\theta_2) = {\bf 0}$ if and only if
$\Psi$ can be represented as in \eqref{decomp2} for some
real-valued functions $V$ and $W$ satisfying the orthogonality conditions
\eqref{projections2stat}.

By Cauchy--Schwarz inequality, since $\partial_x U, U_1, U_2 \in \mathcal{H}$,
we have $\|{\bf f}(0,0,0) \| \le C \varepsilon$ for some positive $\varepsilon$-independent constant $C$.
Furthermore, the Jacobian matrix of the function ${\bf f}$
at the origin $(0,0,0)$ is given by
\begin{eqnarray*}
  D {\bf f}(0,0,0) & \,=\,&
  \begin{pmatrix} \| \partial_x U \|^2_{\mathcal{H}} & 0 & 0 \\
  0 & \| U_1 \|^2_{\mathcal{H}} & 0 \\
  0 & 0 & \| U_2 \|^2_{\mathcal{H}}
  \end{pmatrix} \\
& \phantom{t} &   +   \begin{pmatrix} \langle {\rm Re} \partial_x (\Psi - U), \partial_x U \rangle_{\mathcal{H}} &
-\langle {\rm Im} (\Psi - U)_1, \partial_x U \rangle_{\mathcal{H}} &
-\langle {\rm Im} (\Psi - U)_2, \partial_x U \rangle_{\mathcal{H}} \\
  \langle {\rm Im} \partial_x (\Psi - U), U_1 \rangle_{\mathcal{H}} & \langle {\rm Re} (\Psi - U)_1, U_1 \rangle_{\mathcal{H}} &
  \langle {\rm Re} (\Psi - U)_2, U_1 \rangle_{\mathcal{H}} \\
 \langle {\rm Im} \partial_x (\Psi - U), U_2 \rangle_{\mathcal{H}} & \langle {\rm Re} (\Psi - U)_1, U_2 \rangle_{\mathcal{H}} & \langle {\rm Re} (\Psi - U)_2, U_2 \rangle_{\mathcal{H}},
  \end{pmatrix}.
\end{eqnarray*}
where the subscript $1,2$ denotes the projection to the first or second component of the vectors, respectively.
The first term in $D {\bf f}(0,0,0)$ is a fixed invertible matrix.
The second term in $D {\bf f}(0,0,0)$ is bounded in the matrix norm by $C\varepsilon$
for another positive $\varepsilon$-independent constant $C$. Indeed, for the second and third columns,
these bounds follow by the Cauchy--Schwarz inequality.
For the first column, before applying the Cauchy--Schwarz inequality,
the $x$ derivative can be moved from $\Psi - U$ to $\partial_x U$, $U_1$, and $U_2$
by integration by parts, with the use of smoothness and decay of $\partial_x U$, $U_1$, and $U_2$.
Hence $D {\bf f}(0,0,0)$ is invertible if $\varepsilon$ is small enough
and the norm of the inverse of $D {\bf f}(0,0,0)$ is bounded by a constant
independent of $\varepsilon$.

Finally, it is straightforward to verify
that the second-order derivatives of ${\bf f}$ are uniformly bounded
near $(0,0,0)$ if $\varepsilon$ is small. These observations together imply that there
exists a unique triple $(\alpha,\theta_1,\theta_2)$, in a neighborhood of size
$\mathcal{O}(\varepsilon)$ of $(0,0,0)$, such that ${\bf f}(\alpha,\theta_1,\theta_2)
= {\bf 0}$. Thus the decomposition \eqref{decomp2} and \eqref{projections2stat} holds for
these values of $(\alpha,\theta_1,\theta_2)$. In addition, the above argument shows
that the modulation parameters $(\alpha,\theta_1,\theta_2)$ depend continuously on
$\Psi \in \mathcal{H}$.
\end{proof}

The Cauchy problem for the coupled GP system \eqref{GP} is globally
well-posed for any $\Psi_0 \in \mathcal{D} \cap L^{\infty}(\R)$ \cite{Zhid}.
If $\Psi(t)$ is a solution of \eqref{GP} in $\mathcal{D} \cap L^{\infty}(\R)$ which stays
in a neighborhood of the orbit of the domain walls $U$ for all $t \in \mathbb{R}$, the
modulation parameters $\alpha(t)$, $\theta_1(t)$, and $\theta_2(t)$ given by
the decomposition \eqref{decomposition2} subject to the orthogonality
conditions \eqref{projections2} are continuous functions of time.
The following lemma controls evolution of the modulation parameters according to
the modulation equations.

\begin{Lemma}
\label{difflem}
If $\varepsilon > 0$ is sufficiently small and if $\Psi(t)$ is
a global solution to the coupled GP equations \eqref{GP} in $\mathcal{D} \cap L^{\infty}(\R)$
satisfying, for all $t \in \R$,
\begin{equation}
\label{inf3}
  \inf_{\alpha, \theta_1,\theta_2 \in \R} \| \Psi_{\alpha,\theta_1,\theta_2} - U \|_{\mathcal{H}} \,\le\, \varepsilon,
\end{equation}
then the modulation parameters $\alpha(t)$, $\theta_1(t)$, and $\theta_2(t)$ in the
decomposition \eqref{decomposition2} and \eqref{projections2}
are continuously differentiable functions of $t$ satisfying
\eqref{bound-time-per}.
\end{Lemma}

\begin{proof}
If $\Psi(t)$ is a global solution to the coupled GP equations \eqref{GP}
in $\mathcal{D} \cap L^{\infty}(\R)$, it is easy to verify that the map $t \mapsto \Psi(t)$
is continuously differentiable in the topology of $H^{-1}(\R)$.
Thanks to the smoothness and decay of $\partial_x U$, $U_1$, and $U_2$,
for all $(\alpha,\theta_1,\theta_2) \in \R^3$, the scalar products
\[
 \langle {\rm Re}(\Psi_{\alpha(t),\theta_1(t),\theta_2(t)}(t)-U), \partial_x U \rangle_{\mathcal{H}}, \;
  \langle {\rm Im} \Psi_{\alpha(t),\theta_1(t),\theta_2(t)}(t), U_1 \rangle_{\mathcal{H}}, \;
  \langle {\rm Im} \Psi_{\alpha(t),\theta_1(t),\theta_2(t)}(t), U_2 \rangle_{\mathcal{H}}
\]
are continuously differentiable functions of time. Thus, if
assumption \eqref{inf3} holds for all times, the proof of
Lemma~\ref{Lemma-xith} shows that the modulations parameters
$\alpha(t)$, $\theta_1(t)$, and $\theta_2(t)$ in the decomposition \eqref{decomposition2} and
\eqref{projections2} are $C^1$ functions of time.

Differentiating both sides of \eqref{decomposition2} and using \eqref{GP},
we obtain the evolution system
\[
  \left\{\!\!\begin{array}{l}
  ~\,\,V_t \,=\, L_- W + \dot{\alpha} (\partial_x U + \partial_x V) - \dot{\theta}_1 W_1 - \dot{\theta}_2 W_2 + \mathcal{E}_-(V,W), \\
  -W_t \,=\, L_+ V -\dot{\alpha} \partial_x W - \dot{\theta}_1 (U + V)_1 - \dot{\theta}_2 (U + V)_2 + \mathcal{E}_+(V,W), \end{array} \right.
\]
where the operators $L_\pm$ are defined in \eqref{L-plus} and \eqref{L-minus}
and $\mathcal{E}_{\pm}(V,W)$ contain quadratic and cubic terms in $(V,W)$,
which are not important for further estimates.
Using the orthogonality conditions \eqref{projections2}, we
eliminate the time derivatives $V_t$ and $W_t$ by taking the corresponding projections in $\mathcal{H}$.
This gives the following linear system for the derivatives
$\dot{\alpha}$, $\dot{\theta}_1$ and $\dot{\theta}_2$:
\begin{equation}
\label{Bsys}
  B \begin{pmatrix} \dot{\alpha} \\ \dot{\theta}_1 \\ \dot{\theta}_2 \end{pmatrix}
  \,=\, \begin{pmatrix} \langle L_- W, \partial_x U \rangle_{\mathcal{H}} \\
  \langle L_+ V, U_1 \rangle_{\mathcal{H}} \\  \langle L_+ V, U_2 \rangle_{\mathcal{H}} \end{pmatrix} \,+\,
\begin{pmatrix}
  \langle \mathcal{E}_-(V,W), \partial_x U \rangle_{\mathcal{H}} \\
  \langle \mathcal{E}_+(V,W), U_1 \rangle_{\mathcal{H}} \\
   \langle \mathcal{E}_+(V,W), U_2 \rangle_{\mathcal{H}}
\end{pmatrix},
\end{equation}
where
\begin{eqnarray}
B =
  \begin{pmatrix} -\| \partial_x U \|^2_{\mathcal{H}} & 0 & 0 \\
  0 & \| U_1 \|^2_{\mathcal{H}} & 0 \\
  0 & 0 & \| U_2 \|^2_{\mathcal{H}}
  \end{pmatrix} +   \begin{pmatrix} -\langle \partial_x V, \partial_x U \rangle_{\mathcal{H}} &
\langle W_1, \partial_x U \rangle_{\mathcal{H}} &
\langle W_2, \partial_x U \rangle_{\mathcal{H}} \\
  \langle \partial_x W, U_1 \rangle_{\mathcal{H}} & \langle V_1, U_1 \rangle_{\mathcal{H}} &
  \langle V_2, U_1 \rangle_{\mathcal{H}} \\
 \langle \partial_x W, U_2 \rangle_{\mathcal{H}} & \langle V_1, U_2 \rangle_{\mathcal{H}} &
 V_2, U_2 \rangle_{\mathcal{H}},
  \end{pmatrix}.
  \label{BBdef}
\end{eqnarray}
As in the proof of Lemma~\ref{Lemma-xith}, it is easy to verify by using
\eqref{inf3} and the Cauchy--Schwarz inequality that the second term in $B$
is bounded in the matrix norm by $C \varepsilon$ for some positive $\varepsilon$-independent constant
$C$. Since the first term in $B$ is a diagonal matrix with nonzero entries independently of $\varepsilon$,
the matrix $B$ is invertible with an uniformly bounded inverse if $\varepsilon$ is small enough.

Let us show that the second term in the right-hand side of \eqref{Bsys}
is of size $\mathcal{O}(\varepsilon^2)$. It is sufficient to consider few particular quadratic and cubic terms
in $\langle \mathcal{E}_+(V,W), U_1 \rangle_{\mathcal{H}}$ such as
$\langle UV^2, U_1 \rangle_{\mathcal{H}}$ and $\langle V^3, U_1 \rangle_{\mathcal{H}}$.
For the quadratic term, we obtain by integration by parts
\begin{eqnarray*}
\left| \int_{\mathbb{R}} (\partial_x u_1 v_1^2) (\partial_x u_1) dx \right| =
\left| \int_{\mathbb{R}} u_1 v_1^2  u_1'' dx \right| \leq
C \int_{\mathbb{R}} u_1 v_1^2 (1- u_1^2) dx \leq C \| V \|_{\mathcal{H}}^2 \leq C \varepsilon^2,
\end{eqnarray*}
where we have used the bound $|u_1''(x)| \leq C (1-u_1^2(x))$ for every $x \in \mathbb{R}$ and some $C > 0$,
that follows from properties (b) and (d). Similarly, we have
\begin{eqnarray*}
\left| \int_{\mathbb{R}} (1-u_1^2) u_1^2 v_1^2 dx \right| \leq C \| V \|_{\mathcal{H}}^2 \leq C \varepsilon^2.
\end{eqnarray*}
For the cubic term, we obtain by using the same bound for $|u_1'(x)|$ and the Cauchy--Schwarz inequality
\begin{eqnarray*}
\left| \int_{\mathbb{R}} (\partial_x v_1^3) (\partial_x u_1) dx \right|
& \leq & C \int_{\mathbb{R}} v_1^2 |\partial_x v_1| (1- u_1^2) dx \\
& \leq & C \| v_1 (1 - u_1^2)^{1/2} \|_{L^{\infty}} \int_{\mathbb{R}} |v_1| |\partial_x v_1| (1- u_1^2)^{1/2} dx \\
& \leq & C \| v_1 (1 - u_1^2)^{1/2} \|_{L^{\infty}} \| \partial_x v_1 \|_{L^2} \| (1- u_1^2)^{1/2} v_1 \|_{L^2} \\
& \leq & C \| V \|_{\mathcal{H}}^3 \leq C \varepsilon^3,
\end{eqnarray*}
where we have used the Sobolev embedding $\| (1 - u_1^2)^{1/2} v_1 \|_{L^{\infty}} \leq C \| (1 - u_1^2)^{1/2} v_1  \|_{H^1}$
and the elementary inequality
$$
\| (1 - u_1^2)^{1/2} v_1 \|_{H^1}^2 \leq \| \partial_x v_1 \|_{L^2}^2 + \left\| \frac{u_1 u_1'}{(1-u_1^2)^{1/2}} v_1 \right\|_{L^2}^2
+ \| (1-u_1^2)^{1/2} v_1 \|_{L^2}^2 \leq C \| V \|_{\mathcal{H}}^2,
$$
due to the same bound for $|u_1'(x)|$. Similarly, we obtain
\begin{eqnarray*}
\left| \int_{\mathbb{R}} (1-u_1^2) v_1^3 u_1 dx \right| & \leq &
C \int_{\mathbb{R}} v_1^2 |\partial_x v_1| (1- u_1^2) dx \\
& \leq & C \| (1 - u_1^2)^{1/2} v_1 \|_{L^{\infty}} \| (1- u_1^2)^{1/4} v_1 \|_{L^2}^2 \\
& \leq & C \| V \|_{\mathcal{H}}^3 \leq C \varepsilon^3,
\end{eqnarray*}
where we have used for every $\alpha > 0$ and every $x_0 \in \mathbb{R}$ that
$$
\| (1- u_1^2)^{\alpha} v_1 \|_{L^2(-\infty,x_0)} \leq \| v_1 \|_{L^2(-\infty,x_0)} \leq \| V \|_{\mathcal{H}}
$$
and
$$
\| (1- u_1^2)^{\alpha} v_1 \|_{L^2(x_0,\infty)} \leq C \| V \|_{\mathcal{H}},
$$
where the latter bound is due to the exponential decay of $1 - u_1^2(x)$ to zero as $x \to +\infty$ and
the slow growth of $v_1(x)$ as follows
$$
|v_1(x)| \leq |v_1(x_0)| + \| \partial_x v_1 \|_{L^2} |x-x_0|^{1/2} \leq C \| V \|_{\mathcal{H}} (1 + |x-x_0|^{1/2}).
$$
By using similar estimates for other quadratic and cubic terms, we verify that
the second term in the right-hand side of \eqref{Bsys} is of size $\mathcal{O}(\varepsilon^2)$.
On the other hand, the first term in the right-hand side of \eqref{Bsys} is of size $\mathcal{O}(\varepsilon)$
by the Cauchy--Schwarz inequality.

It follows from \eqref{Bsys} and \eqref{BBdef} by inverting $B$ and estimating the right-hand-side as above
that $|\dot{\alpha}(t)| + |\dot{\theta}_1(t)| + |\dot{\theta}_2(t)| \le C\varepsilon$ for all $t \in \R$,
where the positive $\varepsilon$-independent constant $C$ is also independent of $t$. This concludes
the proof of the bound \eqref{bound-time-per}.
\end{proof}

\section{Proof of Theorem \ref{Theorem-main}}

The energy estimates of Section 4 and the modulation equations of Section 5 are sufficient for the proof
of Theorem \ref{Theorem-main}. If $\Psi(t)$ is a solution of \eqref{GP} in $\mathcal{D} \cap L^{\infty}(\R)$
starting with the initial data $\Psi_0 \in \mathcal{D} \cap L^{\infty}(\R)$, which is close to the domain
walls in the sense of the bound (\ref{bound-initial}), then we introduce the modulation parameters
according to the decomposition (\ref{decomposition2}) and (\ref{projections2}) which are defined
by Lemma \ref{Lemma-xith} at least for small values of time $t$. Then, thanks to the translation and gauge
invariance, we define the conserved energy function
\begin{equation}
E(\Psi) = E(\Psi_{\alpha(t),\theta_1(t),\theta_2(t)}) = E(U+V+iW)
\end{equation}
and use the energy estimates (\ref{control}) to control the proximity of the solution
from the domain wall $U$ in the sense of the distance (\ref{distance-R}). By the estimate (\ref{inital-energy}),
we can choose $\delta = \mathcal{O}(\nu e^{-3 \kappa R})$ in the initial bound (\ref{bound-initial}),
where $\nu$ is defined in (\ref{ball}). Then, by (\ref{control}), we can choose
$\varepsilon = \mathcal{O}(\nu e^{-3 \kappa R})$ in the bound (\ref{bound-final}) for all times.
This construction extends the definition of modulation parameters $\alpha$, $\theta_1$, and $\theta_2$
in the decomposition (\ref{decomposition2}) and (\ref{projections2}) to all times. Then, Lemma \ref{difflem}
yields the control of the evolution of the modulation parameters with at most linear growth in
time $t$, according to \eqref{bound-time-per}. Theorem \ref{Theorem-main} is proved.

\appendix

\section{Continuation of eigenvalues in the spectral problem (\ref{spectral-problem-R})}

In this appendix, we prove the continuity of eigenvalues of the spectral problem (\ref{spectral-problem-R})
as $R \to \infty$. This result is needed for the proof of coercivity of the operator $L_R$ in
$\mathcal{H}$ subject to a single orthogonality condition, see bound (\ref{coercivity-L-R}) in
Lemma \ref{lem-coercivity-L-R}.
For reader's convenience, we write the spectral problem (\ref{spectral-problem-R}) again:
\be\label{1}
L_R \Psi = \lambda K \Psi, \quad \Psi \in \mathcal{H}.
\ee
The formal limit as $R \to \infty$ is given by the spectral problem
(\ref{spectral-problem-L-plus}), which is written as
\be\label{2}
L_+ \Psi = \lambda K \Psi, \quad \Psi \in H^1(\mathbb{R}).
\ee

The following theorem ensures that the isolated eigenvalues of the spectral problem
(\ref{2}) below the point $\lambda_0 = 1$ are continued as the eigenvalues of
the spectral problem (\ref{1}) for sufficiently large $R > 0$.

\renewcommand{\theTheorem}{A}

\begin{Theorem}\label{Theo1}
For some $N \in \N$, suppose that the spectral problem \eqref{2} has the first $N$
smallest eigenvalues below $\lambda_0 = 1$, which are ranked in the ascending order as follows:
\be\label{3}
\lambda_1^{\infty} < \lambda_2^{\infty} \le  \dots \le \lambda_N^{\infty} < 1,
\ee
counting by multiplicity. Then, for $R > 0$ sufficiently large,
the spectral problem \eqref{1} has the first $N$ smallest eigenvalues below $\lambda_0 = 1$,
which are also ranked in the ascending order as follows:
\be\label{4}
\lambda_1^{R} < \lambda_2^{R} \le  \dots \le \lambda_N^R < 1,
\ee
with the following convergence property
\be\label{5}
\lim\limits_{R \to \infty} \lambda_n^{R} = \lambda_n^{\infty} \quad (n = 1, \dots, N).
\ee
Moreover, if $\Phi_n^R = {{\varphi_n^R}\choose{\psi_n^R}}  \in \mathcal{H} ~ (n = 1, \dots, N)$
denotes eigenfunctions associated with $\lambda_1^R, \dots, \lambda_n^R$, normalized by
\be\label{6}
\langle \Phi_n^R, \Phi_m^R \rangle_{\mathcal{H}} = \delta_{nm},
\ee
then there exist linearly independent eigenfunctions $\Phi_n^{\infty} = {{\varphi_n^{\infty}}\choose{\psi_n^{\infty}}}  \in
H^1(\R) ~ (n  = 1, \dots, N)$ associated with $\lambda_1^{\infty}, \dots, \lambda_n^{\infty}$ such that
\be\label{7}
\Phi_n^{R_j} \stackrel{j \to \infty}{\rightharpoonup} \Phi_n^{\infty} \text{ weakly in } \mathcal{H} \quad (n = 1, \dots, N)
\ee
for some sequence $R_j \to \infty$.
\end{Theorem}

The proof of the theorem is subdivided into several Lemmas. We denote by $\mathcal{Q}_R$
and $\mathcal{Q}_{\infty}$ the bilinear forms (on $\mathcal{H}$ and $H^1(\R)$, resp.) defined
by the left-hand sides of problem \eqref{1} and \eqref{2}, respectively. We also introduce the following
matrix potentials:
$$
M_- := \begin{bmatrix} u_1^2 + \gamma u_2^2 - 1 & 0 \\ 0 & \gamma u_1^2 + u_2^2 - 1 \end{bmatrix}, \quad
M_+ := 2\begin{bmatrix} u_1^2 & \gamma u_1 u_2\\  \gamma u_1 u_2 & u_2^2 \end{bmatrix},
$$
and
$$
M_K := (\gamma-1) \begin{bmatrix} 1-u_1^2 & 0 \\ 0 & 1-u_2^2 \end{bmatrix}.
$$

\renewcommand{\theLemma}{A.1}

\begin{Lemma}\label{Lem1}
For $R$ sufficiently large, problem \eqref{1} has at least $N$ eigenvalues below $\lambda_0 = 1$,
ordered as in \eqref{4}, and we have
\be\label{8}
\limsup\limits_{R \to \infty} \lambda_n^R \le \lambda_n^{\infty} \quad (n = 1, \dots, N).
\ee
\end{Lemma}

\begin{proof}
For $n = 1, \dots, N$, define
\be\label{9}
\lambda_n^R :=
\inf\limits_{\scriptsize{\begin{array}{c} U \subset \mathcal{H} \text{ subspace} \\ \text{dim} U = n \end{array}}} ~
          \max\limits_{\Phi \in U \setminus \{ 0 \}} ~
          \frac{\mathcal{Q}_R  (\Phi, \Phi)}{\langle \Phi, \Phi \rangle_{\mathcal{H}}} ,
\ee
and let $\varepsilon \in (0,2)$ be fixed. For $R$ sufficiently large and any $\Phi : {{\varphi}\choose{\psi}} \in H^1(\R)$,
we have on $[R,+\infty)$:
\bn
\Phi^T M_+ \Phi \ge 2 \varphi^2 - \varepsilon (\varphi^2 + \psi^2) & \ge & - \varepsilon \psi^2 \ge - 2 \varepsilon (1-u_2^2) \psi^2 \\
& \ge & - \frac{2\varepsilon}{\gamma - 1} \left[ (\varphi')^2 + (\psi')^2 + (\varphi,\psi) M_K {{\varphi}\choose{\psi}} \right],
\en
and an analogous inequality on $(-\infty, -R]$. Hence, for all $\Phi \in H^1(\R) \setminus \{ 0 \}$ and $R$ sufficiently large,
\bn
\frac{\mathcal{Q}_R  (\Phi, \Phi)}{\langle \Phi, \Phi \rangle_{\mathcal{H}}} \le \frac{\mathcal{Q}_{\infty}  (\Phi, \Phi)}{\langle \Phi, \Phi \rangle_{\mathcal{H}}}  + \frac{2 \varepsilon}{\gamma - 1}.
\en
So the min-max-principle gives
\be\label{10}
\inf\limits_{\scriptsize{\begin{array}{c} U \subset H^1(\R) \text{ subspace} \\ \text{dim} U = n \end{array}}} ~
          \max\limits_{\Phi \in U \setminus \{ 0 \}} ~
          \frac{\mathcal{Q}_R  (\Phi, \Phi)}{\langle \Phi, \Phi \rangle_{\mathcal{H}}}
          \le \lambda_n^{\infty} + \frac{2 \varepsilon}{\gamma - 1}
\ee
for $n = 1, \dots, N$ and $R$ sufficiently large. Since $\mathcal{H} \supset H^1(\R)$,
then \eqref{9} and \eqref{10} imply that
\be\label{11}
\lambda_n^R \le \lambda_n^{\infty} + \frac{2 \varepsilon}{\gamma - 1} \quad (n = 1, \dots, N)
\ee
for $R$ sufficiently large. When $\varepsilon$ is small enough (such that the right-hand side of \eqref{11} is
less than $1$ for $n = N$), the min-max-principle shows that $\lambda_1^R, \dots, \lambda_N^R$ are indeed
the $N$ smallest {\it eigenvalues} of problem \eqref{1}, since the spectrum of \eqref{1} is discrete
for any $R > 0$. Finally, \eqref{11} gives
\bn
\limsup\limits_{R \to \infty} \lambda_n^R \le \lambda_n^{\infty} + \frac{2 \varepsilon}{\gamma - 1} \quad (n = 1, \dots, N)
\en
and hence the claim \eqref{8} since $\varepsilon \in (0,2)$ is arbitrary.
\end{proof}

\renewcommand{\theLemma}{A.2}
\begin{Lemma}\label{Lem2}
Suppose that for some sequence $(R_j) \to \infty$, the limits
\be\label{12}
\hat{\lambda}_n = \lim\limits_{j \to \infty} \lambda_n^{R_j} \quad (n = 1, \dots, N)
\ee
exist. Then, $\hat{\lambda}_1, \dots, \hat{\lambda}_N$ are eigenvalues of problem \eqref{2}, and
\be\label{13}
\hat{\lambda}_n \ge \lambda_n^{\infty} \quad (n = 1, \dots, N).
\ee
Moreover, with $\Phi_n^R \in \mathcal{H} ~ (n = 1, \dots, N)$ denoting eigenfunctions associated with $\lambda_n^R ~ (n = 1, \dots, N)$, normalized by \eqref{6}, there exist linearly independent eigenfunctions $\hat{\Phi}_1, \dots, \hat{\Phi}_N \in H^1(\R) ~ (n = 1, \dots, N)$ associated with $\hat{\lambda}_1, \dots, \hat{\lambda}_N$ such that, for some subsequence $(R_{j_k})$,
\be\label{14}
\Phi_n^{R_{j_k}}  \stackrel{k \to \infty}{\rightharpoonup} \hat{\Phi}_n \text{ weakly in } \mathcal{H} \quad (n = 1, \dots, N).
\ee
\end{Lemma}

\begin{proof}
By \eqref{6}, the sequence $(\Phi_n^{R_{j}})_{j \in \N}$ is bounded in the Hilbert space $\mathcal{H}$ for each 
$n \in \{ 1, \dots, N\}$, whence $\hat{\Phi}_1, \dots, \hat{\Phi}_N \in \mathcal{H}$ exist such that \eqref{14} holds. We will show that
\be\label{15}
\hat{\Phi}_1, \dots, \hat{\Phi}_N \in H^1(\R)
\ee
and that
\be\label{16}
\hat{\Phi}_1, \dots, \hat{\Phi}_N \text{ are linearly independent}
\ee
in the subsequent Lemmas \ref{Lem3}, \ref{Lem4}, and \ref{Lem5}.

Fix $n \in \{ 1, \dots, N\}$ and $\Psi \in C_c^{\infty}(\R)$, and $R_0 > 0$ such that $\text{supp}\; \Psi \subset (-R_0, R_0)$.
Since \eqref{14} implies that
\bn
\left( \Phi_n^{R_{j_k}} \right)' \rightharpoonup \hat{\Phi}_n', \quad
\Phi_n^{R_{j_k}} \rightharpoonup \hat{\Phi}_n \quad \text{ weakly in } L^2 (-R_0,R_0),
\en
we obtain, for $k$ such that $R_{j_k} \ge R_0$,
\be\label{17}
\mathcal{Q}_{R_{j_k}} \left( \Phi_n^{R_{j_k}}, \Psi \right) = \mathcal{Q}_{\infty} \left( \Phi_n^{R_{j_k}},\Psi \right) \underset{\scriptsize{k \to \infty}}{\rightarrow} \mathcal{Q}_{\infty} \left( \hat{\Phi}_n , \Psi \right)
\ee
and
\be\label{18}
\langle \Phi_n^{R_{j_k}},\Psi \rangle_{\mathcal{H}} \underset{\scriptsize{k \to \infty}}{\rightarrow} \langle \hat{\Phi}_n,\Psi \rangle_{\mathcal{H}}.
\ee
Since $(\lambda_n^{R_{j_k}}, \Phi_n^{R_{j_k}})$ is an eigenpair of problem \eqref{1}, and moreover $\lambda_n^{R_{j_k}} \to \hat{\lambda}_n ~ (k \to \infty)$, \eqref{17} and \eqref{18} imply
\be\label{19}
\mathcal{Q}_{\infty} \left( \hat{\Phi}_n,\Psi \right) = \hat{\lambda}_n \langle  \hat{\Phi}_n,\Psi \rangle_{\mathcal{H}}.
\ee
This holds for every $\Psi \in C_c^{\infty} (\R)^2$, and hence by \eqref{15} for every $\Psi \in H^1(\R)$. Thus, \eqref{16} (implying $\hat{\Phi}_n \not\equiv 0$) and \eqref{19} show that $(\hat{\lambda}_n,\hat{\Phi}_n)$ is indeed an eigenpair of problem \eqref{2}.

Finally, \eqref{4} and \eqref{12} show that $\hat{\lambda}_1 \le \dots \le \hat{\lambda}_N$, which by \eqref{16} implies the claim \eqref{13} since $\lambda_1^{\infty} \le \dots \le \lambda_N^{\infty}$ are the $N$ {\it smallest} eigenvalues of problem \eqref{2}.
\end{proof}

{\it Proof of Theorem \ref{Theo1}:} Fix $n_0 \in \{ 1, \dots, N\}$, and choose some sequence $(R_j) \to \infty$ such that
\be\label{20}
\lambda_{n_0}^{R_j} \underset{\scriptsize{j \to \infty}}{\rightarrow} \liminf\limits_{R \to \infty} \lambda_{n_0}^{R} =: \hat{\lambda}_{n_0}.
\ee
It is easy to check that $M_- + M_R + \frac{2 \gamma + 1}{\gamma -1} M_K$ is positive semi-definite
on $[-R,R]$ for every $R > 0$, whence
\bn
\lambda_n^R \ge - \frac{2 \gamma + 1}{\gamma -1} \quad (n = 1, \dots, N).
\en
So $\left( \lambda_n^{R_j} \right)_{j \in \N}$ is bounded for all $n \in \{ 1, \dots, N\}$, whence along a subsequence, denoted by $(R_j)$ again, $\lambda_n^{R_j}$ converges to some $\hat{\lambda}_n$, for each $n \in \{ 1, \dots, N\} \setminus \{ n_0\}$. Using Lemma \ref{Lem2}, property \eqref{13} together with \eqref{20} shows that
\bn
\liminf\limits_{R \to \infty} \lambda_{n_0}^R \ge \lambda_{n_0}^{\infty}.
\en
This holds for every $n_0 \in \{ 1, \dots, N\}$, which together with Lemma \ref{Lem1} proves the claim \eqref{5}.

By \eqref{5}, the assumption \eqref{12} of Lemma \ref{Lem2} holds for $\hat{\lambda}_n := \lambda_n^{\infty} ~ (n =  1, \dots, N)$, and hence \eqref{14} implies \eqref{7} with $\Phi_n^{\infty} := \hat{\Phi}_n ~ (n = 1, \dots, N)$. \hfill \qed\\

The next three Lemmas provide the proof of properties (\ref{15}) and (\ref{16}).

\renewcommand{\theLemma}{A.3}
\begin{Lemma}\label{Lem3}
Property \eqref{15} holds.
\end{Lemma}

\begin{proof}
Fix $n \in \{ 1, \dots, N\}$ and let $\hat{\Phi}_n = {{\hat{\varphi}_n}\choose{\hat{\psi}_n}}$. Since $\hat{\Phi}_n \in \mathcal{H}$, we are left to show that
\be\label{21}
\hat{\varphi}_n \mid_{(0,\infty)} \in L^2 (0,\infty) , ~ \hat{\psi}_n \mid_{(-\infty,0)} \in L^2 (-\infty,0).
\ee
The orthonormal property \eqref{6} and Lemma \ref{Lem1} show that for sufficiently large $R > 0$, we have
\be\label{22}
\mathcal{Q}_R \left( \Phi_n^R, \Phi_n^R \right) = \lambda_n^R \langle \Phi_n^R, \Phi_n^R \rangle_{\mathcal{H}} = \lambda_n^R \le \lambda_n^{\infty} + 1.
\ee
On the other hand, denoting $\Phi_n^R = {{\varphi_n^R}\choose{\psi_n^R}}$, we obtain
\bn
\mathcal{Q}_R\left( \Phi_n^R, \Phi_n^R \right) & \ge & \int\limits_{\R}  \left( \Phi_n^R \right)^T M_-  \Phi_n^R dx + \int\limits_{-R}^R \left( \Phi_n^R \right)^T M_+ \Phi_n^R dx\\
&\ge& - \frac{1}{\gamma -1} \langle \Phi_n^R, \Phi_n^R \rangle_{\mathcal{H}} + \int\limits_{-R}^R [ 2 u_1^2 (\varphi_n^R)^2 + 4 \gamma u_1 u_2 \varphi_n^R \psi_n^R + 2 u_2^2 (\psi_n^R)^2] dx\\
&\ge& - \frac{1}{\gamma -1} + \int\limits_0^R [ 2 u_1^2 (\varphi_n^R)^2 - u_1^2 ( \varphi_n^R)^2 - 4 \gamma^2 u_2^2 ( \psi_n^R)^2 +  2 u_2^2 (\psi_n^R)^2] dx\\
&& \phantom{textte} + \int\limits_{-R}^0 [ 2 u_1^2 (\varphi_n^R)^2 - 4 \gamma^2 u_1^2 ( \varphi_n^R)^2 -  u_2^2 (\psi_n^R)^2  + 2 u_2^2 (\psi_n^R)^2 ] dx.
\en
The right hand side is now estimated from below by
\bn
&& - \frac{1}{\gamma -1} +  \left( \min\limits_{[0,\infty)} u_1^2 \right) \int\limits_{0}^R (\varphi_n^R)^2 dx +
\left( \min\limits_{(-\infty,0]} u_2^2 \right) \int\limits_{-R}^0 (\psi_n^R)^2 dx\\
&& - (4 \gamma^2 - 2) \left[  \left(  \max\limits_{[0,\infty)} \frac{u_2^2}{1-u_2^2} \right)  \int\limits_{0}^R  (1-u_2^2) (\psi_n^R)^2 dx
+  \left(  \max\limits_{(-\infty,0]} \frac{u_1^2}{1-u_1^2} \right)  \int\limits_{-R}^0  (1-u_1^2) (\varphi_n^R)^2 dx   \right].
\en
Since here the two minima are positive and the two maxima are finite, and since
\bn
 \int\limits_{0}^R  (1-u_2^2) (\psi_n^R)^2 dx +  \int\limits_{-R}^0  (1-u_1^2) (\varphi_n^R)^2 dx \le \frac{1}{\gamma -1} \langle \Phi_n^R, \Phi_n^R \rangle_{\mathcal{H}}  = \frac{1}{\gamma - 1},
\en
we obtain together with \eqref{22} that there exists an $R$-independent positive constant $C$ such that
\be\label{23}
\int\limits_{0}^R (\varphi_n^R)^2 dx \le C, \quad \int\limits_{-R}^0 (\psi_n^R)^2 dx \le C
\ee
for all sufficiently large $R$.

Now fix some $R_0 > 0$. Since weak convergence in $\mathcal{H}$ implies strong convergence in $L^2(-R_0,R_0)$, we obtain from \eqref{14} that
\bn
\varphi_n^{R_{j_k}} \underset{\scriptsize{k \to \infty}}{\rightarrow} \hat{\varphi}_n \text{ in } L^2 (0,R_0), \quad
\psi_n^{R_{j_k}} \underset{\scriptsize{k \to \infty}}{\rightarrow} \hat{\psi}_n \text{ in } L^2 (-R_0,0),
\en
and thus for $k$ such that $R_{j_k} \ge R_0$, using \eqref{23},
\bn
\| \hat{\varphi}_n \|_{L^2 (0,R_0)} = \lim\limits_{k \to \infty} \| \varphi_n^{R_{j_k}} \|_{L^2 (0,R_0)}  \le \limsup\limits_{k \to \infty}  \| \varphi_n^{R_{j_k}} \|_{L^2 (0,R_{j_k})}  \le C,
\en
and analogously $\| \hat{\psi}_n \|_{L^2 (-R_0,0)} \le C$. Since this holds for every $R_0 >0$, the claim \eqref{21} follows.
\end{proof}

\renewcommand{\theLemma}{A.4}
\begin{Lemma}\label{Lem4}
(auxiliary for Lemma \ref{Lem5}): Let $\eta \in (0,1)$. Then some $x_0 > 0$ exists such that, for all ${{\varphi}\choose{\psi}} \in \mathcal{H}$ satisfying $\varphi (x_0) = \psi (x_0) = 0$, and all $R \ge x_0$,
\begin{eqnarray}\label{24}
&& \int\limits_{x_0}^{\infty}  \left\{ (\varphi')^2 + (\psi')^2 + (\varphi,\psi) (M_- + M_+ \chi_{[-R,R]}) {{\varphi}\choose{\psi}}
\right\} dx \ge \nonumber\\
&& \ge (1-\eta) \int\limits_{x_0}^{\infty} \left\{ (\varphi ')^2 + (\psi')^2 + (\varphi,\psi) M_K  {{\varphi}\choose{\psi}}  \right\} dx .
\end{eqnarray}
\end{Lemma}

\begin{proof}
The asserted inequality is equivalent to
\begin{eqnarray}
\label{25}
\eta  \int\limits_{x_0}^{\infty} \{ (\varphi ')^2 + (\psi')^2 \} dx & \ge & \int\limits_{x_0}^{\infty} \{ (1-\eta) (\gamma - 1) (1-u_1^2) -
(u_1^2 + \gamma u_2^2 - 1)  \} \varphi^2 dx \nonumber\\
& \phantom{t} & + \int\limits_{x_0}^{\infty} \{  (1-\eta) (\gamma - 1) (1-u_2^2) - (\gamma u_1^2 + u_2^2 - 1) \} \psi^2 dx \nonumber\\
& \phantom{t} & + 2  \int\limits_{x_0}^{R} \{- u_1^2 \varphi^2 - u_2^2 \psi^2 - 2 \gamma u_1 u_2 \varphi \psi \} dx.
\end{eqnarray}
Since the three integrands on the right-hand side of \eqref{25} are bounded from above by
\bn
\gamma (1-u_1^2) \varphi^2, ~ \gamma (1-u_1^2) \psi^2, \text{ and } \gamma u_1 u_2 (\varphi^2 + \psi^2),
\en
respectively, and since $1-u_1^2 \le 2 (1-u_1)$ and $u_1 \le 1$, the following inequality is sufficient for \eqref{25}:
\be\label{26}
\eta \int\limits_{x_0}^{\infty} \{ (\varphi ')^2 + (\psi')^2 \} dx \ge 2 \gamma \int\limits_{x_0}^{\infty} (1-u_1 + u_2) ( \varphi^2  + \psi^2) dx.
\ee
We know from properties (a) and (d) of the domain wall solutions
that there exist some positive constants $C$ and $\alpha$ such that
\be\label{27}
2 \gamma \{ 1 - u_1(x) + u_2(x) \} \le C e^{-\alpha x} \quad \mbox{\rm for all \;} x > 0.
\ee
Finally, for ${{\varphi}\choose{\psi}} \in \mathcal{H}$ satisfying $\varphi (x_0) = \psi (x_0) = 0$, and all $y \ge x_0$,
\bn
 \int\limits_{x_0}^{y} e^{-\alpha x} \varphi^2 dx & = & - \frac{1}{\alpha} e^{-\alpha x} \varphi^2 \Big|_{x_0}^y + \frac{2}{\alpha} \int\limits_{x_0}^{y} e^{-\alpha x} \varphi \varphi' dx \\
 &\le&  \frac{2}{\alpha} e^{-\frac{\alpha}{2}  x_0} \int\limits_{x_0}^{y}  e^{-\frac{\alpha}{2} x} \mid \varphi \varphi' \mid dx \\
 &\le&  \frac{1}{\alpha} e^{-\frac{\alpha}{2} x_0} \left[  \int\limits_{x_0}^{y}  e^{-\alpha x} \varphi^2 dx +  \int\limits_{x_0}^{\infty} ( \varphi')^2 dx \right]
 \en
 and hence, if $\frac{1}{\alpha} e^{-\frac{\alpha}{2} x_0} < 1$,
\bn
 \int\limits_{x_0}^{y}  e^{-\alpha x} \varphi^2 dx \le \frac{\frac{1}{\alpha}  e^{-\frac{\alpha}{2} x_0} }{1- \frac{1}{\alpha}  e^{-\frac{\alpha}{2} x_0} } \int\limits_{x_0}^{\infty} (\varphi')^2 dx.
\en
Thus, the integral on the left converges as $y \to \infty$.

An analogous inequality holds with $\psi$ instead of $\varphi$. Together with \eqref{27} we find that \eqref{26}, and hence \eqref{25} holds if $x_0$ is large enough to satisfy
\bn
 \frac{ \frac{C}{\alpha}  e^{-\frac{\alpha}{2} x_0} }{1- \frac{1}{\alpha}  e^{-\frac{\alpha}{2} x_0} } \le \eta.
\en
Thus, the claim (\ref{24}) follows.
\end{proof}

\renewcommand{\theLemma}{A.5}
\begin{Lemma}\label{Lem5}
Property \eqref{16} holds.
\end{Lemma}

\begin{proof}
Suppose for contradiction that some non-trivial $(\alpha_1, \dots, \alpha_N) \in \CC^N$ exists such that
\be\label{28}
\sum\limits_{n=1}^N \alpha_n \hat{\Phi}_n \equiv 0.
\ee
W.l.o.g. let $\sum\limits_{n=1}^N |\alpha_n|^2 = 1$. Using the subsequence $(R_{j_k})$ satisfying \eqref{14}, we define
\be\label{29}
\Phi^{(k)} := \sum\limits_{n=1}^N \alpha_n \Phi_n^{R_{j_k}} \quad (k \in \N),
\ee
whence \eqref{14} and \eqref{28} imply
\be\label{30}
\Phi^{(k)}  \underset{\scriptsize{k \to \infty}}{\rightharpoonup} 0 \text{ weakly in } \mathcal{H}.
\ee
Furthermore, using \eqref{29} and \eqref{6},
\be\label{31}
\langle \Phi^{(k)} , \Phi^{(k)}  \rangle_{\mathcal{H}} = \sum\limits_{n=1}^N |\alpha_n|^2 = 1.
\ee
Choose
\be\label{32}
\eta := \frac{1}{4} ( 1 - \hat{\lambda}_N),
\ee
which by \eqref{12}, \eqref{8}, and \eqref{3} is positive. Now choose $x_0$ according to Lemma \ref{Lem4}. Besides \eqref{24}, an analogous inequality also holds with integration over $(-\infty,-x_0)$ instead of $(x_0,\infty)$, possibly after further enlargening $x_0$.

We define
\bn
S(x) &:=& \left\{
\begin{array}{ll}
0            & ( |x| \le x_0) \\
\sin [ \frac{\pi}{2} (|x| - x_0) ]& (x_0 \le |x| \le x_0 + 1) \\
1            & (|x| \ge x_0 + 1) \end{array}  \right\} , \\[2ex]
C (x) &:=& \left\{ \begin{array}{ll}
1            & ( |x| \le x_0)\\
\cos [\frac{\pi}{2} (|x| - x_0) ]& (x_0 \le |x| \le x_0 + 1) \\
0            & (|x| \ge x_0 + 1) \end{array}  \right\} ~.
\en
Since $S \Phi^{(k)} \in \mathcal{H}$ vanishes on $[-x_0,x_0]$, \eqref{24} (and the analogous inequality over $(-\infty,-x_0)$) implies, for all $R \ge x_0$ and $k \in \N$,
\be\label{33}
\mathcal{Q}_R (S \Phi^{(k)}, S \Phi^{(k)}) \ge (1-\eta) \langle S \Phi^{(k)}, S \Phi^{(k)} \rangle_{\mathcal{H}}.
\ee
Furthermore, denoting $I_0 := [- x_0 - 1, - x_0 ] \cup [x_0, x_0 + 1]$,
\bn
S^2 + C^2 \equiv 1, \quad (S')^2 + (C')^2 = \frac{\pi^2}{4} \chi_{I_0} \text{ on } \R,
\en
and therefore, for all $R \ge x_0$ and $k \in \N$,
\begin{eqnarray}
\label{34}
\mathcal{Q}_R(S \Phi^{(k)}, S \Phi^{(k)}) + \mathcal{Q}_R(C \Phi^{(k)}, C \Phi^{(k)}) =
\mathcal{Q}_R   ( \Phi^{(k)},  \Phi^{(k)})  + \frac{\pi^2}{4} \int\limits_{I_0} | \Phi^{(k)} |^2 dx,
\end{eqnarray}
and
\begin{eqnarray}
\label{35}
&&\langle S \Phi^{(k)}, S \Phi^{(k)} \rangle_{\mathcal{H}} +  \langle C \Phi^{(k)}, C \Phi^{(k)} \rangle_{\mathcal{H}}
= \langle  \Phi^{(k)},  \Phi^{(k)} \rangle_{\mathcal{H}}  + \frac{\pi^2}{4} \int\limits_{I_0} | \Phi^{(k)} |^2 dx \ge 1,
\end{eqnarray}
using \eqref{31} in the last step.

By compact embedding, \eqref{30} implies $\Phi^{(k)} \to 0$ stongly in $L^2(-x_0-1,x_0+1)$, and hence
\be\label{36}
C \Phi^{(k)}  \to 0 \text{ strongly in } L^2(\RR).
\ee

{\bf Case I:} $\| (C \Phi^{(k_{\nu})} )' \|_{L^2 (\R)^2 }\ge \delta > 0$ along some subsequence.

Then, together with \eqref{36}, we obtain
\bn
\frac{ \mathcal{Q}_R (C \Phi^{(k_{\nu})}, C \Phi^{(k_{\nu})} ) }{\langle C \Phi^{(k_{\nu})},C \Phi^{(k_{\nu})} \rangle_{\mathcal{H}} }
 \underset{\scriptsize{\nu \to \infty}}{\rightarrow} 1, \text{ uniformly in } R,
\en
and therefore, for $\nu$ sufficiently large,
\bn
\mathcal{Q}_R  (C \Phi^{(k_{\nu})}, C \Phi^{(k_{\nu})} ) \ge (1-\eta) \langle C \Phi^{(k_{\nu})}, C \Phi^{(k_{\nu})} \rangle_{\mathcal{H}}
\en
for all $R \ge x_0$. Together with \eqref{33}, \eqref{34}, \eqref{35} this implies
\bn
\mathcal{Q}_R  ( \Phi^{(k_{\nu})},  \Phi^{(k_{\nu})} ) +  \frac{\pi^2}{4} \int\limits_{I_0} | \Phi^{(k_{\nu})} |^2 dx \ge 1- \eta
\en
and thus, using again that $\Phi^{(k_{\nu})} \to 0$ in $L^2 (- x_0 - 1, x_0+1)^2$,
\be\label{37}
\mathcal{Q}_R  ( \Phi^{(k_{\nu})},  \Phi^{(k_{\nu})} ) \ge 1-2\eta
\ee
for $\nu$ sufficiently large, uniformly in $R \in [x_0, \infty)$. On the other hand, by \eqref{29} and \eqref{6},
\begin{eqnarray}
\label{38}
\mathcal{Q}_{R_{j_{k_{\nu}}}}   ( \Phi^{(k_{\nu})},  \Phi^{(k_{\nu})} ) & = &
\sum\limits_{n,m = 1}^N  \alpha_n \overline{\alpha}_m \mathcal{Q}_{R_{j_{k_{\nu}}}}
( \Phi^{R_{j_{k_{\nu}}}}_n,  \Phi^{R_{j_{k_{\nu}}}}_m ) \nonumber\\
& = & \sum\limits_{n,m = 1}^N  \alpha_n \overline{\alpha}_m \lambda_n^{R_{j_{k_{\nu}}}} \delta_{nm} \le \lambda_N^{R_{j_{k_{\nu}}}},
\end{eqnarray}
which contradicts \eqref{37} due to \eqref{12} and \eqref{32}.
\bigskip

{\bf Case II:} $(C \Phi^{(k)})' \to 0$ in $L^2 (\R)^2$.

Then, using also \eqref{36}, we obtain
\bn
\mathcal{Q}_R (C \Phi^{(k)},C \Phi^{(k)}) \to 0, \quad \langle C \Phi^{(k)}, C \Phi^{(k)} \rangle_{\mathcal{H}} \to 0
\en
as $k \to \infty$, uniformly in $R \in [x_0, \infty)$. Therefore, using \eqref{34}, \eqref{35}, and the convergence $ \Phi^{(k)} \to 0$ in $L^2 (- x_0-1, x_0 + 1)$,
\bn
\mathcal{Q}_R (S \Phi^{(k)},S \Phi^{(k)}) \le \mathcal{Q}_R (  \Phi^{(k)},  \Phi^{(k)} )+ \eta,
\en
and
\bn
\langle S \Phi^{(k)}, S \Phi^{(k)}  \rangle_{\mathcal{H}} \ge 1 - \eta
\en
for $k$ sufficiently large, uniformly in $R$. Together with \eqref{33}, this gives
\be\label{39}
\mathcal{Q}_R ( \Phi^{(k)}, \Phi^{(k)}) \ge (1-\eta)^2- \eta \ge 1 - 3 \eta
\ee
for $k$ sufficiently large, uniformly in $R \in [x_0, \infty)$. On the other hand, as in the calculation \eqref{38}, we obtain
\bn
\mathcal{Q}_{R_{j_k}} ( \Phi^{(k)}, \Phi^{(k)}) \le \lambda_N^{R_{j_k}}
\en
which contradicts \eqref{39}, again due to \eqref{12} and \eqref{32}.
\end{proof}

\end{document}